\documentclass[12pt]{article}
\usepackage{amssymb}\usepackage{amsmath}\usepackage{amsthm}
\usepackage{hyperref}
\usepackage{xypic}

\renewcommand{\geq}{\geqslant}
\renewcommand{\leq}{\leqslant}

\newcommand{\om}{\omega}
\newcommand{\la}{\langle}
\newcommand{\ra}{\rangle}
\newcommand{\nin}{\not\in}
\newcommand{\A}{\forall}
\newcommand{\E}{\exists}
\newcommand{\noi}{\noindent}
\newcommand{\subs}{\subseteq}

\newcommand{\ms}{\medskip}
\newcommand{\bs}{\bigskip}
\newcommand{\s}{\sigma}

\newcommand{\N}{\mathbb{N}}

\newcommand{\emp}{\emptyset}

\renewcommand{\b}{\bullet}
\renewcommand{\a}[1]{\ar@{-}[#1]}

\newtheorem{thm}{Theorem}[section]
\newtheorem{lem}[thm]{Lemma}
\newtheorem{cor}[thm]{Corollary}

\theoremstyle{definition}
\newtheorem{defn}[thm]{Definition}

\title{Self-embeddings of computable trees}

\author{Stephen Binns \and Bj\o rn Kjos-Hanssen \and Manuel Lerman 
\and James H. Schmerl \and Reed Solomon\footnote{The authors thank Carl Jockusch and Ted Slaman for helpful discussions regarding the results in the last section of this article. Solomon's research was partially supported by an NSF Grant DMS-0400754.}
}

\begin{document}

\maketitle

\begin{abstract}
We divide the class of infinite computable trees into three types.  For the first and second types, $0'$ computes a nontrivial 
self-embedding while for the third type $0''$ computes a nontrivial self-embedding.  These results are optimal and we obtain partial results concerning the complexity of 
nontrivial self-embeddings of infinite computable trees considered up to isomorphism.  We show that every infinite computable tree 
must have either an infinite computable chain or an infinite $\Pi^0_1$ antichain.  This result is optimal and has connections to the program of reverse mathematics.
\end{abstract}

\section{Introduction}
\label{intro}

In this article, we examine self-embeddings of countable trees from the perspective of computable algebra.  The following definition of a tree 
is more restrictive than some other definitions in the literature but it is more general than the notion of computable tree used 
in the context of $\Pi^0_1$ classes.

\begin{defn}
A \textbf{tree} is a partial order $(T, \preceq)$ with a least element (called the \textbf{root} of $T$ and denoted by $\lambda$) such that 
for all $n \in T$, the set $\{ m \in T \mid m \preceq n \}$ is a finite linearly ordered set.  The elements of $T$ are referred to as \textbf{nodes}. 
If $n \prec m$ and there are no elements strictly between $n$ and $m$ in $T$, we say that $m$ is a \textbf{successor} of $n$.  If $n$ 
has no successor, then it is called a \textbf{leaf} and if $n$ has more than one successor, then it is called a \textbf{branching node}.  
If $(T_0, \preceq_0)$ and $(T_1, \preceq_1)$ are trees, then an \textbf{embedding} from $T_0$ to 
$T_1$ is an injective function $f: T_0 \rightarrow T_1$ such that $n \preceq_0 m$ if and only if $f(n) \preceq_1 f(m)$.   We write 
$f:T_0 \hookrightarrow T_1$ to denote that $f$ is an embedding of $T_0$ into $T_1$.  
\end{defn}

Our concern is with countable trees, so we assume $|T| \leq \omega$ for the rest of this article.  The \textbf{branching 
function} $\text{br}:T \rightarrow \omega \cup \{ \infty \}$ of $T$ is the function which maps $n \in T$ to the number of successors of $n$.  Notice 
that $n$ is a leaf if and only if $\text{br}(n) = 0$.  $T$ is \textbf{finitely branching} if 
the range of the branching function is contained in $\omega$ and $T$ is \textbf{binary branching} if the range of 
the branching function is contained in $\{ 0,1,2 \}$.  

Classically, any countable tree $T$ is isomorphic to a subtree of $\omega^{< \omega}$.   
A tree $(T,\preceq)$ is \textbf{computable} if $T \subseteq \omega$ 
is a computable set and $\preceq$ is a computable relation defined on $T^2$.  (It does not hurt to assume that a computable tree is coded in such 
a way that $T = \omega$.)  For example, if $T$ is a computable nonempty subset of $2^{< \omega}$ or $\omega^{< \omega}$ 
which is closed under initial segments and $\preceq$ is the initial segment relation, then $(T, \preceq)$ is a computable tree.  In these cases, 
the successor relation is computable and in the case when $T \subseteq 2^{< \omega}$, the leaf relation and the branching function 
are computable.  However, in general, it is not the case that the successor relation, the leaf relation or 
the branching function is computable for a computable tree or 
that a computable tree is computably isomorphic to a computable subtree of $\omega^{< \omega}$.  

A large amount of work has been done on $\Pi^0_1$ classes, which are sets of infinite paths through computable subtrees of $2^{< \omega}$ and 
$\omega^{< \omega}$.  (See Cenzer \cite{cen:99} and Cenzer and Remmel \cite{cen:98} for surveys of this work.)  
In addition, work has been done on the possible degrees of isomorphism types of trees by Richter \cite{ric:81} and on 
computable categoricity of trees by R.~Miller \cite{mil:05} and by Lempp, McCoy, R.~Miller and Solomon \cite{lem:05}.  One of the main tools for working with 
trees in \cite{lem:05}, \cite{mil:05} and \cite{ric:81} and in our current work is Kruskal's Lemma.  

\begin{lem}[Kruskal \cite{kru:60}]
Let $\{ T_i \mid i \in \omega \}$ be a countable collection of finite trees.  There exists $k \in \omega$ such that for all $i > k$, there are infinitely many 
$j > i$ for which $T_i$ embeds into $T_j$.  
\end{lem}

The main motivation for the present work comes from the effective analysis of the Dushnik--Miller Theorem \cite{dus:40}.  
This theorem states that any countably infinite 
linear order has a nontrivial self-embedding.  (A self-embedding of a linear order is called nontrivial if it moves at least one, and hence infinitely many, points.)  
Downey and Lempp \cite{dow:99} analyzed the classical proof of the Dushnik--Miller Theorem and 
observed that if $L$ is a computable infinite linear order then $0''$ computes a nontrivial self-embedding of $L$.  
It remains an open question whether there is such an order which requires $0''$ to compute a nontrivial self-embedding, but 
they showed that there is a computable linear order $L$ for which any nontrivial self-embedding computes $0'$.  

Downey, Jockusch and J.S.~Miller \cite{dow:ta} carried this analysis further and showed that for computable discrete linear orders,
\footnote{A discrete linear order is a linear order in which every element except the least 
has an immediate predecessor and every element except the greatest has an immediate successor.} 
being PA over $0'$ is enough to compute a nontrivial 
self-embedding and conversely that there is a computable discrete linear order for which every nontrivial self-embedding is PA over $0'$. 
\footnote{A Turing degree is called PA over $0'$ if it can compute an infinite path through any infinite $0'$-computable subtree of $2^{< \omega}$.}   
They also proved that there is an infinite nondiscrete computable linear order for which $0'$ cannot compute a nontrivial self-embedding.

Our goal is to carry out a similar analysis for nontrivial self-embeddings of computable trees.  In the context of trees, it is useful to take a slightly 
more restrictive definition of nontrivial.  We say that a self-embedding $f:T \rightarrow T$ is \textbf{nontrivial} if it is not an onto map. (It follows from the proof 
of the Dushnik--Miller Theorem that every countably infinite linear order has a self-embedding which is nontrivial in this sense as well.)  
We say that a self-embedding $f$ is \textbf{weakly nontrivial} if it is not the identity map.  In the context of trees (as opposed to linear orders), this condition 
does not imply that $f$ must move infinitely many nodes.  However, frequently one can build trees for which any weakly nontrivial self-embedding 
moves infinitely many points.  We will state our results for nontrivial self-embeddings and occasionally point out cases in which they can also be 
applied to weakly nontrivial self-embeddings.    

It is known that every countably infinite tree has a nontrivial self-embedding and such an embedding can be obtained by a simple application 
of Kruskal's Lemma.  (For example, see Ross \cite{ros:89}.  This existence result 
also holds for more general definitions of trees using extensions of Kruskal's Lemma 
such as the main theorem from Corominas \cite{cor:85}.)  In Section \ref{sec:upper}, we analyze such a proof 
and show that $0''$ computes a nontrivial self-embedding for any computable infinite tree $T$.  This proof naturally breaks into three cases 
depending on the structure of $T$ and we use this distinction to define three classes of trees.

If $n$ is a node in $T$, we let $T(n)$ denote the subtree $\{ m \mid n \preceq m \}$ with the inherited partial order.   We say that 
$n$ is an \textbf{infinite node} if $T(n)$ is infinite and we say that $n$ is an $\mathbf{\omega}$\textbf{-node} if $n$ has infinitely many successors.  A 
\textbf{path} in $T$ is an infinite maximal linearly ordered subset of $T$.  We say that a path $P \subseteq T$ is \textbf{isolated} if there is a node 
$n \in P$ such that $P$ is the only path containing $n$.  $T$ is called a \textbf{type 1 tree} if it contains a maximal infinite node.  That is, $T$ contains a 
node $n$ such that $T(n)$ is infinite but for all $m \succ n$, $T(m)$ is finite.  $T$ is called a \textbf{type 2 tree} if it does not have a maximal 
infinite node but does have an isolated path.  $T$ is called a \textbf{type 3 tree} if it is infinite but has no maximal infinite node and no isolated paths.  
(R.~Miller \cite{mil:05} used the same classification of height $\omega$ trees as well as proof techniques similar to several of those used here.)    
We show that if $T$ is an infinite computable tree of type 1 or type 2, then $0'$ can compute a nontrivial self-embedding of $T$ and that if $T$ is 
an infinite computable type 3 tree, then $0''$ can compute a nontrivial self-embedding of $T$.  

There are at least two questions that one might ask concerning the optimality of these results.  First, is there a computable type 1 (or type 2) tree $T$ for 
which every nontrivial self-embedding computes $0'$?  Second, is there is computable type 1 (or type 2) tree $T$ such that for every 
(classically) isomorphic computable tree $S \cong T$ and for every nontrivial self-embedding $f$ of $S$, $f$ computes $0'$?  (There are similar questions 
concerning the connection between $0''$ and nontrivial self-embeddings of type 3 trees.)  When answering the first 
question, one is allowed to use facts about the particular computable coding (or presentation) of $T$, while in the second question, one must work 
with the isomorphism type of $T$ and not with the particular coding.  We examine both questions with respect to each type of trees.  

In Section \ref{sec:type1}, we show that there is a computable type 1 tree for which every nontrivial self-embedding computes $0'$ and 
there is a computable type 1 tree for which no (classically) isomorphic computable tree has a computable nontrivial self-embedding.  In Section 
\ref{sec:type2}, we show the same results for computable type 2 trees.  In Section \ref{sec:type3}, we show that there is a computable type 3 tree for 
which every nontrivial self-embedding computes $0''$ and there is a computable type 3 tree such that no (classically) isomorphic computable tree 
has a $0'$-computable nontrivial self-embedding.  (We actually show something slightly stronger but we leave the technical statement of the result 
until Section \ref{sec:type3}.)  

These results show that the bounds of $0'$ and $0''$ are optimal in the sense of specific computable trees but we only obtain partial results in 
terms of the isomorphism types.  It remains an open question whether there exists a type 1 or 2 computable tree for which $0'$ is necessary to 
compute a nontrivial self-embedding in every isomorphic computable copy and whether there exists a type 3 computable tree for which $0''$ is 
necessary to compute a nontrivial self-embedding in every isomorphic computable copy.  

In Section \ref{sec:cac}, we turn to a slightly different question.  The Chain/Anti-Chain Principle states that every infinite partial order has 
either an infinite chain or an infinite antichain.  (This principle is a simple consequence of Ramsey's Theorem for pairs.)  Herrmann \cite{her:01} investigated this 
result from the perspective of computable combinatorics and proved that every infinite computable partial order has either an infinite $\Delta^0_2$ chain or an 
infinite $\Pi^0_2$ antichain.  Furthermore, he proved that these bounds were optimal by constructing an infinite computable partial order which has 
no infinite $\Sigma^0_2$ chains or antichains.  

We examine how these results can be improved in the context of trees as a special type of partial order.  
We show that every infinite computable tree has either an infinite computable chain or an infinite $\Pi^0_1$ antichain.  
Furthermore, we show that these bounds are optimal by constructing an infinite computable tree which has no infinite $\Sigma^0_1$ chains or antichains.  
Our construction is easily modified to work in the context of models of the subsystem $\text{WKL}_0$ of second order arithmetic.  Thus we show that 
$\text{WKL}_0$ is not strong enough to prove the Chain/Anti-Chain Principle for binary branching trees. 

Our computability theoretic notation is standard and follows Soare \cite{soa:book}.  In particular, we use $\varphi_e$ to denote the $e^{\text{th}}$ 
partial computable function, we use $K$ to denote the halting set (or any other complete computably enumerable set), and we use $X[n]$ to 
denote $\{ m \in X \mid m < n \}$ for any set $X$.  The relation $\preceq$ denotes a tree order, $\leq$ denotes the standard order on 
$\omega$ and $\leq_T$ denotes Turing reducibility.  

\section{Nontrivial self-embeddings}
\label{sec:upper}

In this section we show that $0''$ suffices to compute a nontrivial self-embedding of any infinite computable tree and that for certain special 
cases, $0'$ is sufficient.  For any tree $T$, let $S(m,n)$ denote the successor relation on $T$ (that 
$m$ is a successor of $n$) and let $\text{br}: T \rightarrow \omega \cup \{ \infty \}$ denote the branching function on $T$.   
As mentioned in the introduction, the successor relation and the branching function need not be computable even if $T$ is computable.  

\begin{lem}
\label{universal}
Every computable tree $T$ embeds into $2^{<\om}$ and the embedding is computable in $0'$.
\end{lem}

\begin{proof} 
Given an arbitrary $n\in T$ we compute the image of $n$ as follows. $0'$ computes the successor relation of $T$ 
(which is explicitly $\Pi_1^0$) so we can use  $0'$ to compute the sequence 
$$\lambda = n_0\prec n_1 \prec \dots \prec n_k=n,$$
where $S(n_{i+1}, n_i)$ for all $i$. The image of $n$ will then be 
$${1^{n_0}}\ast 0\ast {1^{n_1}}\ast 0\ast \dots {1^{n_k}}\ast 0.$$
(where $1^{m}$ denotes the string of $m$ ones and $*$ denotes concatenation). It is straightforward to confirm that this gives an embedding.  
(Recall that an embedding does not need to be closed under initial segments.)
\end{proof}

We defined a path in $T$ to be an infinite maximal linearly ordered subset of $T$.  Often we specify a path by giving an infinite 
sequence of elements $x_0 \prec x_1 \prec x_2 \prec \cdots$ such that $S(x_{i+1}, x_i)$.  (This method of specifying a path corresponds 
to the notion of path used in the study of $\Pi^0_1$ classes in which the successor relation is computable.)  From such a sequence, we can 
compute the path $P = \{ y \mid \exists i (y \preceq x_i) \}$ by calculating the elements $x_i$ in order until we find an $x_i$ such that 
either $y \preceq  x_i$ or $y$ is incomparable with $x_i$.  This procedure cannot necessarily be reversed; in general, we cannot 
compute such a sequence from a given path $P$ because the successor relation need not be computable from $P$.  However, we can 
compute a sequence of elements $y_0 \prec y_1 \prec \cdots$ which is cofinal in $P$ without requiring that $S(y_{i+1},y_i)$ holds.  
Furthermore, from such a sequence we can also compute the associated path.  

\begin{lem}
\label{maxomega}
If $T$ is an infinite tree and it has no maximal infinite node then it has a path.  
\end{lem}

\begin{proof}
Suppose $T$ is an  infinite tree and has no maximal infinite node. If $T$ has no $\omega$-node then it is finitely branching and has a 
path by K\"onig's Lemma. Otherwise, $T$ must have an $\omega$-node $n_0$ and a successor $n_1$ of $n_0$ such that $T(n_1)$ 
is infinite (otherwise $n_0$ would be a maximal infinite node). If $T(n_1)$ has no $\omega$-nodes, then $T(n_1)$ has a path (and so does $T$) 
by K\"{o}nig's Lemma.  Otherwise, 
$T(n_1)$ must have an $\omega$-node $n_2$ which has a successor $n_3$ such that $T(n_3)$ is infinite. Iterating in this way, we 
either arrive at an infinite finitely branching tree $T(n_i)$ which has a path, or we obtain an infinite sequence 
$n_0\prec n_1\prec n_2 \prec \dots$ which defines a path.  
\end{proof}

\begin{thm}
\label{thm:ntse}
Every infinite computable tree has a nontrivial self-embedding computable in $0''$.
\end{thm}

The proof splits (nonuniformly) into three cases depending on whether the infinite computable tree is of type 1, 2 or 3.  By Lemma \ref{maxomega}, if 
$T$ is infinite but not of type 1, then $T$ must have a path, so either $T$ has an isolated path (and is a type 2 tree) or $T$ has no 
isolated path (and is a type 3 tree).  Therefore, every infinite tree is either type 1, 2 or 3.  The next three lemmas cover these 
cases and show that for computable trees of type 1 or 2, $0'$ computes a nontrivial self-embedding.  

\begin{lem}
\label{lem:ntse1}
Every computable type 1 tree has a nontrivial self-embedding computable in $0'$.
\end{lem}

\begin{proof}
Let $T$ be a computable type 1 tree with maximal infinite node $n$.  
Because $T(n)$ is infinite but $T(m)$ is finite for all $m \succ n$, $n$ must be an 
$\omega$-node.   Let $n_0<n_1<\dots $ be the successors of $n$. (These are computable from $0'$.)  By Kruskal's Lemma, there is 
a $k$ such that  
\[
\A m\geq k \E^\infty s\geq m ( T(n_m) \hookrightarrow T(n_s)).
\]
Fix such a $k$. We define a $0'$-computable embedding $\varphi$ such that for all $m$, $\varphi(m)\neq m$ if and  only if $m\succeq n_l$ for some $l\geq k$.  
Furthermore, for all $l\geq k$, $\varphi(n_l)=n_j$ for some $j> l$.  

To define $\varphi$, we set $\varphi(m) = m$ for all $m \in T$ such that $n \not \preceq m$ or $m=n$ or $n_i \preceq m$ for some $i < k$.  
We define $\varphi$ on the subtrees $T(n_j)$ for $k \leq j$ by induction on $j$.  
Fix $j\geq k$ and suppose that $\varphi$ has been defined on $T(n_i)$ for all $i< j$.  We let $T_s$ denote the subtree formed by restricting $\preceq$ to 
$\{ 0, 1, \ldots, s \}$.  Use $0'$ to find an $s$ and $t$ such that 
\begin{enumerate}
\item $n_t > \max\{\varphi(n_i): i < j\} + 1$ (the max is taken with respect to $\leq$), 
\item $T_s(n_j)$ embeds into $T_s(n_t)$,
\item $ \A s'>s\ T_{s'}(n_j)=T_s(n_j)$.
\end{enumerate}

Such an $s$ and $t$ exist by our choice of $k$ and because $n$ is a maximal infinite node and hence $T(n_i)$ is finite for each $i$. We extend 
$\varphi$ to include the embedding of $T(n_j) = T_s(n_j)$ into $T_s(n_t)$.  Because $\varphi(n_i) = n_i$ for $i < k$, we have that for all $j \geq k$, if 
$\varphi(n_j) = n_t$, then $t > k$.  Therefore $n_k$ is not in the range of $\varphi$, so $\varphi$ is nontrivial.  
\end{proof}

If $n$ is a successor of the root of a tree $T$, then we say $T(n)$ is a \textbf{successor tree} of $T$.  Similarly, if $n$ is a successor of $m$ in 
$T$, then we say that $T(n)$ is a \textbf{successor tree of} $T(m)$.  

\begin{lem}
\label{lem:ntse2}
Every computable type 2 tree has a nontrivial self-embedding computable in $0'$.
\end{lem}

\begin{proof}
Let $T$ be a computable type 2 tree.  By definition, $T$ has no maximal infinite node and has an isolated path $X$.  
Fix $i\in T$ such that there is only one infinite path extending $i$. By Lemma \ref{maxomega}, if $j\succ i$ is any node on $X$ there is exactly one 
successor $j'$ of $j$ such that $T(j')$ is infinite, namely the successor that is on $X$. Therefore the only possible $\omega$-nodes extending $i$ lie on 
$X$.  If $m\succ i$ is an $\omega$-node, all but one of its successor trees are finite and we are essentially in the case of a type 1 tree. (Let $n_0 < n_1 < \cdots$ 
be the successor nodes of $m$ and let $l$ be such that $T(n_l)$ is the only infinite successor tree.  Apply the argument in the proof of Lemma 
\ref{lem:ntse1} to the sequence of successor 
nodes $n_{l+1} < n_{l+2} < \cdots$.)  Therefore, we assume there are no $\omega$-nodes above $i$. 

We compute $X$ from $0'$ by computing the sequence $i = x_0 \prec x_1 \prec x_2 \prec \cdots$ such that 
$S(x_{i+1},x_i)$ and each $x_i \in X$.  Suppose we have calculated $x_j$.  We use $0'$ to find a stage $s$ such that 
\[  
\A n\succ x_j \E m\in T_s(x_j) (  m\neq x_j \text{ and } \ n\succeq m).
\]
Such an $s$ much exist as $T(x_j)$ is finitely branching by our assumption.  The finite number of nodes which appear to be the successors of 
$x_j$ at stage $s$ are the actual successors of $x_j$.  Again using $0'$, search for a successor $x'$ of $x_j$ and a $t$ such that 
\[
\A v\geq t \ [v\succ x_j \implies v\succeq x'].
\]
Such $x'$ and $t$ must exist as $x_j$ has exactly one path through it.  Since $x'$ is the successor of $x_j$ on $X$, we set 
$x_{j+1} = x'$.

We have defined (from $0'$) the sequence $i=x_0\prec x_1 \prec x_2 \prec \dots$ such that for all $i$, $x_i$ is on the path $X$ and $S(x_{i+1}, x_i)$.  
We now can apply Kruskal's theorem to the sequence of finite trees  $T(x_i)\smallsetminus T(x_{i+1})$ and use an argument similar to the proof of 
Lemma \ref{lem:ntse1}.
\end{proof}

\begin{lem}
\label{lem:ntse3}
Every computable type 3 tree has a nontrivial self-embedding computable in $0''$.
\end{lem}

\begin{proof}
Let $T$ be an computable type 3 tree.  By definition, $T$ is infinite but has no maximal infinite node and no isolated path.  
We use $0''$ to define an embedding $\varphi: 2^{<\om} \hookrightarrow T$ by recursion (described below) and then use $0'$ to define an 
embedding $\beta: T \hookrightarrow 2^{<\om}$ as in 
Lemma \ref{universal}.  The composition $\alpha = \varphi \circ \beta$ is the desired nontrivial self-embedding $\alpha: T \hookrightarrow T$.  
(There are numerous ways to see that the self-embedding $\alpha$ is nontrivial.   The empty sequence $\emptyset$, which is the root of $2^{< \omega}$, is not mapped 
to the root of $T$ by $\varphi$ and hence $\varphi$ is not onto.  Also, the map $\beta$ from Lemma \ref{universal} is not onto.  Either of these facts 
is enough to conclude that $\alpha$ is nontrivial.)  

We define the embedding $\varphi:2^{<\om}\rightarrow T$ by recursion using $0''$.  
Let $m$ be any node other than the root of $T$ for which $T(m)$ is infinite. By our case assumption, $T(m)$ has no maximal infinite nodes 
and no isolated paths.  Let $\varphi(\emp)=m$.   Assume that $\varphi(\tau)$ has been defined for all $\tau$ such that $|\tau| \leq k$, that $T(\varphi(\tau))$ 
is infinite and that $\varphi$ gives an embedding of $2^{\leq k}$ into $T$.  Consider each $\sigma$ with $|\sigma|=k$ separately and we show how 
to define $\varphi(\sigma*0)$ and $\varphi(\sigma*1)$.  Assume $\varphi(\sigma) = n$.  
Using $0''$ find two incomparable extensions of $n$, say $n_0$ and $n_1$, such that  $T(n_0)$ and  $T(n_1)$ are both infinite.  Because 
$T(n)$ is infinite and $T$ has no isolated 
paths or maximal infinite nodes, such nodes $n_0$ and $n_1$ must exist. Set $\varphi(\s\ast 0)=n_0$ and $\varphi(\s\ast 1)=n_1$. It is easy to check that the 
inductive assumptions hold at level $k+1$.  
\end{proof}

Theorem \ref{thm:ntse} gives an analysis of the existence of nontrivial self-embeddings in terms of the jump hierarchy.  We could also ask for an 
analysis in terms of other computable relations on $T$.  That is, are there natural algebraic relations on $T$ such that we can compute a nontrivial 
self-embedding from these relations?  For computable trees of type 1 or 2, there is a nontrivial self-embedding computable from 
the join of the successor relation and the branching function.  

\begin{cor}
\label{cor:succ1}
If $T$ is a computable type 1 tree, then $T$ has a nontrivial self-embedding computable from the join of the successor relation and the 
branching function.
\end{cor}

\begin{proof}
We used $0'$ twice in the proof of Lemma \ref{lem:ntse1}.  First, we used $0'$ to determine the successors of $n$.  Clearly, 
we can determine these successors from the successor relation.  Second, we used $0'$ to determine if $\forall s' > s (T_{s'}(n_j) = T_s(n_j))$.  
That is, we used it to find a stage by which the finite tree $T(n_j)$ had stopped growing.  Since the trees $T(n_j)$ are finite, we can also determine 
such a stage from the successor relation together with the branching function.   
\end{proof}

\begin{cor}
\label{cor:succ2}
If $T$ is a computable type 2 tree, then $T$ has a nontrivial self-embedding computable from the join of the successor relation and the 
branching function.
\end{cor}

\begin{proof}
Consider the proof of Lemma \ref{lem:ntse2}.  If there is an $\omega$-node $m \succeq i$, then this proof reduced to the proof of Lemma 
\ref{lem:ntse1} so we are done by Corollary \ref{cor:succ1}.  Otherwise, we used $0'$ twice in the definition of the sequence $i = x_0 \prec x_1 \prec \cdots$.  
First, we used it to find a stage 
$s$ such that all the successors of $x_j$ had appeared by stage $s$.  However, if we know the branching function then we can calculate the 
number of successors of $x_j$ (which in this situation is finite) and we can use the successor relation to find this number of successors.  Second, we 
used $0'$ to determine the unique infinite successor of $x_j$.  Because the successor relation together with the branching function can tell when a 
finite tree has stopped growing, we can use them to make this determination as well.  
\end{proof}

The same type of argument does not work in the case of a computable type 3 tree.  
As we will show in Section \ref{sec:type3}, there is an infinite computable tree $T$ which is finitely branching (and hence 
has no maximal infinite nodes), has no isolated paths and for which every nontrivial self-embedding computes $0''$.  We claim that the branching function 
for this tree $T$ is computable from $0'$.  To calculate $\text{br}(n)$, ask $0'$ if there is a node $m \succ n$.  If not, $\text{br}(n) = 0$.  If so, use 
$0'$ to find a successor $n_0$ of $n$.  Ask $0'$ if there is a node $m \succ n$ such that $n_0 \not \preceq m$.  If not, $\text{br}(n) = 1$.  If so, 
use $0'$ to find a second successor $n_1$ of $n$.  Because $T$ is finitely branching, this process must eventually stop with a complete set of 
successors $n_0, \ldots, n_{k-1}$ for $n$.  Therefore, $0'$ can compute both the branching function and the successor relation in $T$.  Since 
every nontrivial self-embedding of $T$ computes $0''$, there cannot be such a self-embedding computable from the join of the successor relation 
and the branching function.  

\section{Type 1 trees}
\label{sec:type1}

Recall that a computable type 1 tree is a computable tree which has a maximal infinite node.  By Lemma \ref{lem:ntse1}, $0'$ computes a nontrivial 
self-embedding for such trees.  In this section, we show that this result is optimal in the sense that there is a computable 
type 1 tree for which every nontrivial self-embedding computes $0'$.  We also show that there is a computable type 1 tree $T$ such that 
for all computable trees $S \cong T$, $S$ does not have a computable self-embedding.  

The \textbf{height of a node} $n$ in $T$ (denoted by $\text{ht}(n)$) is the size of the set $\{ m \mid m \prec n \}$.  For example, the root of 
any tree has height 0 and the successors of the root have height 1.  The \textbf{height of a finite tree} $T$ (denoted $\text{ht}(T)$) is the 
maximum height of a node of $T$.    Frequently, our examples of type 1 trees have an $\omega$ branching root $\lambda$ and have $T(x)$ finite 
for all $x \neq \lambda$.  Recall that for each $x$ of height 1, we say that $T(x)$ is a \textbf{successor tree} of $T$.  

\begin{thm}
\label{maxinf}
There is a computable type 1 tree $T$ such that any nontrivial self-embedding of $T$ computes $0'$.
\end{thm}  
\begin{proof}

Our proof will show that even the weakly nontrivial self-embeddings of $T$ compute $0'$.  
We describe the tree before explicitly constructing it. 0 will be the root of $T$ and the set of successors of 0 will be the set of positive even numbers 
$\mathbb{E}^+$.  Each subtree $T(n)$ with $n\in\mathbb{E}^+$ will be a finite tree with no branching nodes and $\text{ht}(T(n)) \geq n$. (That is, $T(n)$ is a 
finite linear order of length at least $n$.  The exact length of this order will be determined during the construction.)   We construct a sequence of positive even 
numbers $2=m_0<m_1<m_2<\dots $ which will have the following properties:

\begin{itemize}
\item[I.] For all $i \in \omega$, $\text{ht}(T(m_i)) < \text{ht}(T(m_{i+1}))$; 
\item[II.] For all $i \in\om$ and $p, q\in\mathbb{E}^+$, if $m_i \leq p <q< m_{i+1}$ then ht$(T(p))>{\rm ht}(T(q))$ and if $i > 0$, then 
$\text{ht}(T(q)) > \text{ht}(T(m_{i-1}))$;
\item[III.] Let $K=\bigcup_{s=0}^{\infty} K_s$ be a fixed computable enumeration of a complete c.e.~set.  For all $i$, $K[i]=K_{m_i}[i]$.  
(Recall that for any set $X$, $X[i] = \{ n < i \mid n \in X \}$.)  
\end{itemize}

We picture $T$ as looking like a series of strictly descending staircases.  That is, for any $i$, the subtrees $T(m_i), T(m_i+2), T(m_i+4), \ldots, T(m_{i+1}-2)$ 
are all linear orders which are decreasing in height (but all taller than $T(m_{i-1})$).  The subtree $T(m_{i+1})$ jumps up in height (to be taller than 
$T(m_i)$) and begins another sequence of subtrees of decreasing height (but all taller than $T(m_i)$) which continues until we reach $T(m_{i+2})$.  
The idea behind this tree is 
that if $q \in \mathbb{E}^+$ is such that $m_i \leq q < m_{i+1}$, then $T(q)$ does not embed into any subtree of the form $T(p)$ where $p \in \mathbb{E}^+$ 
satisfies $p < m_i$ or $q < p < m_{i+1}$. (It follows from Properties I and II that in both of these cases $\text{ht}(T(q)) > \text{ht}(T(p))$.) 
Therefore, if $\delta$ is a weakly nontrivial self-embedding 
which moves the node $q$, then $\delta$ must map $T(q)$ into a subtree $T(p)$ such that $p \in \mathbb{E}^+$ and either $m_i \leq p < q$ or $m_{i+1} \leq p$.  
In the former case, by iterating $\delta$ we must arrive at some $k \geq 1$ such that $\delta^k$ maps $T(q)$ into $T(p)$ where $p \in \mathbb{E}^+$ and 
$p > q$ (and so $p > m_{i+1}$ by the previous comments).  

Properties I, II and III are sufficient to guarantee that any weakly nontrivial self-embedding $\varphi$ computes $K$.  
Because $\varphi$ is weakly nontrivial, there must be an 
$n\in\om$ such that $\varphi(n)\neq n$.  Fix any such  $n$.  For all $m>0$ let $\lfloor m\rfloor$ be the unique element of 
$\mathbb{E}^+$ such that $m\in T(\lfloor m\rfloor)$.  We define a strictly increasing  function $\psi$ from $\varphi$ as follows:

\ms
\noi$ \psi(0)=\lfloor n\rfloor$

\ms
\noi$ \psi(s+1)= \lfloor \varphi^k(n)\rfloor$ where $k=k(s+1)$ is the least natural number  such that $\lfloor \varphi^k(n)\rfloor > \psi(s)$.

\ms
(The existence of $\psi$ follows from the comments above about the general form of $T$.)  
Once we prove by induction that $\psi(s)\geq m_s$ for all $s\in \om$, we will have by Property III that  $K[i]=K_{\psi(i)}[i]$ and hence that $\psi \geq_T K$.  
Since $\varphi\geq_T \psi$, we have $\varphi\geq_T K$ as claimed. 

The base of the induction is $\psi(0) = \lfloor n \rfloor \geq 2=m_0$.  Suppose that $\psi(s)\geq m_s$. Let $j$ be such that $m_j\leq \psi(s)<m_{j+1}$.  
Then $j\geq s$ as $\la m_s\ra$ is an increasing sequence.  If $j>s$, then we are done as $\psi$ is increasing, so we can assume that 
$m_s \leq \psi(s)<m_{s+1}$. 
  
Property I ensures that for all $l$ if $\psi(s)<l<m_{s+1}$, then $T(\psi(s))\not\hookrightarrow T(l)$.    Therefore, for all $t\geq k=k(s)$, 
\[
\lfloor \varphi^t(n)\rfloor > \psi(s) \implies \lfloor \varphi^t(n)\rfloor\geq m_{s+1}.
\]
But $k(s+1) > k(s)$ and so in particular $\psi(s+1)\geq m_{s+1}$ as required.

It remains to give the construction of a computable type 1 tree $T$ satisfying I, II and III.  We build $T$ in stages.  At stage $s$ we build $T^s$ and $T$ will be 
$\bigcup_s T^s$.  $T^0$ will consist of all the even nodes as above as well as an infinite/coinfinite computable set of odd numbered nodes arranged so 
that for all $e\in\mathbb{E}^+$, ht$(T(e))=e$.  (That is, $T(e)$ is a linear chain of length $e$.)  

We use a movable marker argument to create the sequence $\la m_i\ra$. As we do this we also ensure that I, II and III are satisfied. We describe a uniformly 
computable sequence $\la m_{i,s}\ra$ with the properties
\begin{itemize}
\item[i.] $\A i \, ( m_{i,0}=2i),$
\item[ii.] $\A i, s \, ( m_{i,s}<m_{i+1,s}),$
\item[iii.] $\A i,s \, ( m_{i,s}\leq m_{i,s+1}),$
\item[iv.] $\A i \, ( \lim_s m_{i,s}$ exists).
\end{itemize}
For each $i$, $m_i$ is defined to be $\lim_s m_{i,s}$. We enumerate $K$ one element at a time.  Suppose $s$ is a stage at which 
$i\in K_{s+1}\smallsetminus K_s$ and let $k\geq i$ be the smallest number such that $m_{k,s}\geq s+1$. Then we set

\[ 
m_{j,s+1}=
\begin{cases} 
m_{j,s} & \text{ if } j< i \\
m_{k+t, s} & \text{ if } j=i+t, \  (t \in \om).
\end{cases}
\]
                                                   
At the same time it is necessary to adjust the subtrees $T^s(e)$ with $e\in \mathbb{E}^+$.  We leave all successor trees in $T$ unchanged except perhaps 
those $T^s(e)$ with $m_{i-1,s} = m_{i-1,s+1} \leq e < m_{k,s} = m_{i,s+1}$ (if $i=0$ take $m_{i-1}=2$). To the subtrees $T(e)$ with $m_{i-1,s}\leq e < m_{i,s+1}$, 
we add the minimum number of nodes to the top of each subtree, retaining the property that there are no branching nodes and ensuring that 
Properties I and II are preserved.

The argument that the tree $T$ and the sequence $\la m_i\ra$ have the required properties is now just the typical movable marker argument, made explicit in 
the following lemmas.

 \begin{lem}
For every $j$, $\lim_s m_{j,s} = m_j$ exists.
\end{lem}                                  
                                           
\begin{proof}
Let $s$ be such that $K_s[j]=K[j]$.  Then for all $t\geq s$, if $i\in K_{s+1}\smallsetminus K_s$, then $i> j$, so for all $t\geq s$, $m_{j,t}=m_{j,t+1}$.       
\end{proof}

\begin{lem}
Every successor tree $T(e)$ with $e\in\mathbb{E}^+$  is finite.
\end{lem}

\begin{proof}
Fix $e\in\mathbb{E}^+$. Let $i$ be such that $m_i > e$.  Once $m_{i,s}$ reaches its limit $m_i$, $T^s(e)$ will never grow again.  At each stage 
before this limit is reached, at most finitely many elements are added to $T(e)$.  Therefore, $T(e)$ is finite.  
\end{proof}

\begin{lem}
$T$ has Properties I, II and III.
\end{lem}

\begin{proof} 

The fact that $T$ has Properties I and II follows immediately from the fact that these properties are explicitly retained at each stage in the construction and the 
previous two lemmas.

To see that $T$ satisfies Property III, fix any $j\in \om$. If $s$ is the largest stage at which $i\in K_{s+1}\smallsetminus K_s$ for some $i\leq j$, 
then by the construction: 
\[
m_j= m_{j,s+1}\geq m_{i,s+1} = m_{k,s}\geq s+1,
\]
where $k$ is as in the construction above.
But $K_{s+1}[j]=K[j]$ by our choice of $s$, so $K_{m_j}[j]=K[j].$
\end{proof}

This completes the proof of Theorem \ref{maxinf}.
\end{proof}

The proof of Theorem \ref{maxinf} leads to several corollaries concerning the connection between the complexity of nontrivial self-embeddings of 
computable type 1 trees $T$ and the complexity of natural algebraic relations on $T$.

\begin{cor}
\label{cor:compsucc}
There is a computable type 1 tree $T$ such that $T$ has a computable successor relation and every nontrivial 
self-embedding of $T$ computes $0'$.
\end{cor}  

\begin{proof}
The tree $T$ constructed in Theorem \ref{maxinf} has a computable successor relation because we only add nodes above the top node in $T(e)$ at 
any given stage.  Therefore, if $n,m$ are nodes in $T$ at stage $s$, then $S(n,m)$ holds if and only if it holds at stage $s$.  (Because each nontrivial 
self-embedding of $T$ computes $0'$ and there is such an embedding computable from the join of the successor function and the branching function, 
the branching function for $T$ has degree $0'$.)   
\end{proof}

\begin{cor}
There is a computable type 1 tree $T$ such that $T$ has a computable branching function and every nontrivial 
self-embedding of $T$ computes $0'$.  
\end{cor}

\begin{proof}
This corollary follows by altering the construction in Theorem \ref{maxinf} such that any new nodes which are added to $T(e)$ at stage $s$ 
are placed between the node $e$ (the root of $T(e)$) and its current successor.  By making this change, the nodes which are leaves at stage 0 remain 
leaves throughout the rest of the construction.  Therefore the branching function is equal to 
$\infty$ for the root of $T$ and is equal to 1 for all other nodes except those nodes which are leaves at stage 0.  (By reasoning similar to the 
parenthetical remark at the end of Corollary \ref{cor:compsucc}, the successor relation has degree $0'$ for this tree.)  
\end{proof}

We next turn to the question of working with nontrivial self-embeddings for isomorphism types of computable type 1 trees.  

\begin{thm}
\label{isomaxinf}
There is a computable type 1 tree $T$ such that no computable tree classically 
isomorphic to $T$ has a computable nontrivial self-embedding.
\end{thm}

The remainder of this section is dedicated to the proof of Theorem \ref{isomaxinf}.  
We will build a computable tree $T$ such that the root of $T$ is the only $\omega$-node.  
For each successor $n$ of the root, we call the subtree $T(n)$ a \textit{component} of $T$ 
and we will make each component finite.  Because the root of $T$ will be the only infinite 
node, $T$ will have the required form.  

To build $T$, we uniformly construct the sequence $T_0, T_1, \ldots$ of components of $T$.  (These 
components should also have subscripts indicating the current stage of the construction but 
we suppress the stage subscript unless it is not clear from context.)  Each 
component $T_i$ will have height three and will consist of finitely many components each of which 
will be one of the following four types.

 \ms
 
$\tiny
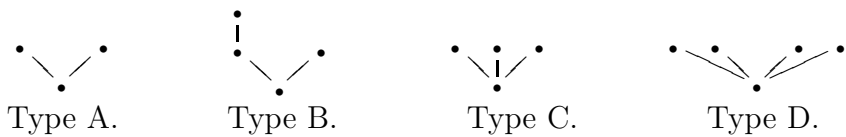
{\xymatrix@-1.5pc{ \\  \b \a{dr} &    & \b \a{dl} \\ & \b }}$ 
\hskip1.3cm
$\tiny{\xymatrix@-1.5pc{  \b \a{d} \\ \b \a{dr} &  & \b \a{dl} \\ & \b}}$
 \hskip1.3cm
$\tiny{\xymatrix@-1.5pc{ \\  \b \a{dr} & \b  \a{d} & \b \a{dl} \\ & \b}}$
 \hskip1.3cm
$\tiny{\xymatrix@-1.5pc{ \\  \b \a{drr} & \b  \a{dr} &  & \b \a{dl} & \b \a{dll} \\ & & \b}}$

  Type A. \hskip1.3cm Type B. \hskip1.3cm \ \ Type C. \hskip1.3cm \ \ Type D.
  
\bs

More specifically, each $T_i$ will contain at most one component of type A, at least one component 
of type B, exactly $i$ many components of type D and finitely many components (including possibly none) of type C.  
Consider the form of a nontrivial self-embedding 
$\delta$ of $T$.  We let $\lambda$ denote the root of $T$ and $r_i$ denote the root of the 
component $T_i$.  Because $T$ has height four, $\delta(\lambda) = \lambda$.  Because each 
$T_i$ has height three, $\delta(r_i) = r_j$ for some $j$.  Because $T_i$ has exactly $i$ many 
type $D$ trees, $\delta(r_i) = r_j$ for some $j \geq i$.  Finally, because each $T_i$ is finite, 
if $\delta$ is nontrivial, then there must be some $i$ for which $\delta(r_i) = r_j$ for $j > i$.  By 
considering iterated images of $T_i$, it is clear that there must be infinitely many indices $i$ 
for which $\delta$ maps $T_i$ into $T_j$ for some $j > i$.  Therefore, any nontrivial 
self-embedding of (any isomorphic copy of) $T$ must map infinitely many components 
into disjoint components.  We will exploit this property in our proof.  

Fix two effective enumerations of the partial computable functions: $\varphi_e$ and $f_i$.  (We use 
two different notations to distinguish between the partial computable function which we view as 
determining the $e$-th tree and the partial computable function which we view as giving a potential 
nontrivial self-embedding of this tree.)  We satisfy all requirements of the form:
\begin{gather*}
R_{\langle e,i \rangle}: \varphi_e \, \text{does not compute a tree isomorphic to} \, T \\
\text{or} \, f_i \, \text{is not a nontrivial self-embedding of the tree computed by} \, \varphi_e.
\end{gather*}

To make these requirements precise, we need to explain how we obtain a tree from $\varphi_e$.  
We view $\varphi_e$ as defining a partial computable relation $\preceq_e$ on universe $\omega$ by 
setting $n \not \preceq_e m$ if $\varphi_e(\langle n,m \rangle)$ converges to $0$ and 
$n \preceq_e m$ if $\varphi_e(\langle n,m \rangle)$ converges to a value other than 0.  If $\varphi_e$ 
is total and $(\omega, \preceq_e)$ is a tree, then we refer to this tree as the $\varphi_e$-tree.  
At a stage $s$, we consider the largest set $X$ such that for all $n,m \in X$, 
$\varphi_{e,s}(\langle n,m \rangle)$ converges and call $(X,\varphi_{e,s})$ the 
$\varphi_e$-tree at stage $s$.  (Of course, this finite structure may have already violated the 
axioms for a tree in which case $\varphi_e$ does not compute a tree and we get an easy win.)  

Because each component $T_i$ of our tree $T$ will contain a type B tree, we know that $T$ 
has height four and each component $T_i$ has height three.  Therefore, if the $\varphi_e$-tree 
at stage $s$ has height greater than four, we know it is not isomorphic to our tree $T$.  Furthermore, 
in the case when the $\varphi_e$-tree is isomorphic to our tree $T$, we can identify the root of 
the $\varphi_e$-tree and the roots of each of its components as it is enumerated.  When we say 
that $S$ is a component of the $\varphi_e$-tree at stage $s$, we mean that the $\varphi_e$-tree 
at stage $s$ has height four, that the least element of $S$ is at level 1 in the $\varphi_e$-tree, that $S$ 
contains a type B component and that $S$ contains all the nodes (currently) above its least element.   

For any component $S$ of the 
$\varphi_e$-tree, we can count the number of type D trees occurring in $S$.  As the component 
$S$ grows, this number can never decrease.  (If we ever see a component of $S$ which 
contains a component larger than a type D tree, 
we know the $\varphi_e$-tree is not isomorphic to our tree 
and we get an easy win.)  Finally, if the component $S$ has $i$ many type D trees at 
stage $s$, then the only possible (current) image of $S$ in $T$ is the component $T_i$.  
As the number of type D trees in $S$ grows, the possible image of this component in our tree 
changes, but if the $\varphi_e$-tree is isomorphic to $T$, then this number must eventually 
stop growing.  Therefore, we eventually correctly guess the only possible image of the component $S$ 
in our tree.  

Before giving the full construction, we consider how to satisfy a single requirement $R_{\langle 
e,i \rangle}$ in isolation.  Assume that $(\omega, \varphi_e)$ is a tree of height four (so that 
we do not get an easy win).  We wait for a stage $s$ at which the $\varphi_e$-tree  
contains disjoint components $U$ and $V$ such that $f_i$ is defined on all of $U$ and embeds 
$U$ into $V$.  (By our analysis of nontrivial self-embeddings of trees isomorphic to $T$, 
if the $\varphi_e$-tree is isomorphic to $T$ and $f_i$ is a nontrivial self-embedding of the 
$\varphi_e$-tree, then such components must exist.)  Once we find such components, we fix them 
and define two parameters.  For all stages $t \geq s$, $u(t)$ is equal to the number of type 
D trees in $U$ at stage $t$ and $v(t)$ is equal to the number of type D trees in $V$ at stage $t$. 
(We also assume that neither $U$ nor $V$ ever adds any additional elements to a type D tree.  
Such additional elements would again give us an easy win.  This assumption means that 
$u(t)$ and $v(t)$ are increasing in $t$.)  
For any stage $t \geq s$, we know that $T_{u(t)}$ is the only possible (current) image of 
$U$ in $T$ and that $T_{v(t)}$ is the only possible (current) image of $V$ in $T$.   

At stage $s$, $u(s) < v(s)$ and $f_i$ gives an embedding of $U$ into $V$.  
We \textit{set up to diagonalize} by taking the following two steps.  First, if $T_{u(s)}$ does not 
have a type A component, then we add elements to $T_{u(s)}$ to create a type A component in 
$T_{u(s)}$.  Second, if $T_{v(s)}$ does have a type A component, then we add an element to 
$T_{v(s)}$ to change this component from type A to type C.  (Recall that $T_{u(s)}$ and $T_{v(s)}$ can have at most one type A component.)  

At each stage $t>s$, we calculate $u(t)$ and $v(t)$ 
and check whether either of these parameters have changed.  Assume for the moment that 
neither of these parameters changes at a future stage.  We next check whether at stage $t$, 
$U$ is isomorphic to $T_{u(t)}$ and $V$ is isomorphic to $T_{v(t)}$.  If not, then we go on to the 
next stage.  If so, we check whether $f_i$ is defined on all of $U$ and is an embedding on 
$U$ into $V$.  If not, go on to the next stage.  If so, then we are ready to diagonalize.  

In this situation, we have $U \cong T_{u(t)}$, $V \cong T_{v(t)}$ and $f_i: U \hookrightarrow V$.  
$T_{u(t)}$ and $U$ each have a single type A component (because of our set up to 
diagonalize procedure) 
and we refer to these components as the \textit{designated components} of $T_{u(t)}$ and $U$.  We refer to the 
component in $V$ which is the image of the designated component in $U$ as the 
designated component in $V$.  Because of our set up to diagonalize procedure, we know that the 
designated component in $V$ is either of type B, C or D.  If the designated component 
of $V$ is of type $B$, then we add an element to $T_{u(t)}$ to change the designated component 
of $T_{u(t)}$ from type A to type C.  If the designated component of $V$ is of type $C$ or $D$, 
then we add an element to $T_{u(t)}$ to change the designated component in $T_{u(t)}$ 
from type A to type B.    

Consider what can happen after this diagonalization step.  If the opponent does not change 
$u(t)$ at a later stage, then the only way for the $\varphi_e$-tree to be isomorphic to 
$T$ is for $U$ to change its designated component (currently of type A) into the same type 
of component as the designated component in $T_{u(t)}$ (which is now either type B or C).  
However, we chose the new type for the designated component in $T_{u(t)}$ in such a way 
that $f_i$ cannot be extended to map from this type of component into the designated 
component in $V$.  Therefore, unless the opponent adds new type D components to $U$ to 
change $u(t)$ or adds new type D components to $V$ to change $v(t)$, we win requirement $R_{\langle e,i \rangle}$.   
If the opponent does change $u(t)$ or $v(t)$, then we repeat this process of setting up to 
diagonalize and later diagonalizing.  

To be more specific about this process and give an indication of the full construction, we remove 
the earlier assumption that $u(t) = u(s)$ and $v(t) = v(s)$ for all $t \geq s$.  
At each stage $t > s$, we calculate $u(t)$ and 
$v(t)$ and check whether $u(t) = u(t-1)$ and $v(t) = v(t-1)$.  If $u(t-1) < u(t)$, then $U$ 
has gained extra type D trees and its potential image in $T$ has changed.  In this 
situation in the full construction, we will take an outcome on a tree of strategies indicating 
that $u(t)$ may be approaching infinity in the limit.  Notice that if $u(t)$ does go to infinity in the 
limit, then $U$ is an infinite component of the $\varphi_e$-tree and we win $R_{\langle e,i 
\rangle}$ because $T$ has no infinite components.  Similarly, if $u(t) = u(t-1)$ but 
$v(t-1) < v(t)$, then in the full construction, we will take an outcome indicating that $v(t)$ may 
be approaching infinity in the limit.  Again, if $v(t)$ goes to infinity in the limit, then $V$ is 
an infinite component of the $\varphi_e$-tree and we win $R_{\langle e,i \rangle}$.  

If $u(t) = u(t-1)$ and $v(t) = v(t-1)$ and $f_i: U \hookrightarrow V$ at stage $t$, then we set up 
diagonalize as above.  That is, we add a type A component to $T_{u(t)}$ (if it does not already 
have one) and change the type A component in $T_{v(t)}$ (if it has one) to a type C component.  
We check if $U \cong T_{u(t)}$ and $V \cong T_{v(t)}$.  If so, then we diagonalize as above.  

There are four possible outcomes of this strategy to meet $R_{\langle e,i \rangle}$ ordered 
from highest to lowest priority by $u_{\infty} <_L v_{\infty} <_L fin <_L triv$.  The tree of 
strategies consists of all finite sequences of these outcomes ordered lexicographically using the 
$<_L$ order.  We use the trivial outcome $triv$ before we have defined the components 
$U$ and $V$ and if the $\varphi_e$-tree has height greater than four or enumerates a component 
$S$ which contains a component larger than a type D tree.  (If we always take this outcome, 
then we win $R_{\langle e,i \rangle}$ because of the form of $T$ and 
the form of the nontrivial self-embeddings of any tree isomorphic to $T$.)  We use 
the $u_{\infty}$ outcome whenever the parameter $u(s)$ increases.  (If we take this outcome 
infinitely often, we win $R_{\langle e,i \rangle}$ because $U$ is an infinite component in 
the $\varphi_e$-tree.)  We use the $v_{\infty}$ outcome whenever the parameter $u(s)$ 
retains its value but the parameter $v(s)$ increases.  (If we take this outcome infinitely often, 
then we win because $V$ is an infinite component in the $\varphi_e$-tree.)  We use the 
outcome $fin$ when both parameters $u(s)$ and $v(s)$ stay the same.  (If we take the 
$u_{\infty}$ and $v_{\infty}$ outcomes only finitely often and the $fin$ outcome infinitely often, 
then we win $R_{\langle e,i \rangle}$ by the diagonalization process.)  

We will use $\alpha$ and $\beta$ to denote nodes on the tree of strategies.  If $|\alpha| = 
\langle e,i \rangle$, then $\alpha$ works on requirement $R_{\langle e,i 
\rangle}$ and we use $\varphi_{\alpha}$ and $f_{\alpha}$ to denote $\varphi_e$ and 
$f_i$.  We denote the components chosen by $\alpha$ in the $\varphi_{\alpha}$-tree by 
$U_{\alpha}$ and $V_{\alpha}$ (which also have stage number subscripts which we typically 
suppress) and we denote the parameters associated with $\alpha$ at 
stage $s$ by $u(\alpha,s)$ and $v(\alpha,s)$.  Each strategy will keep two other parameters 
$a(\alpha,s)$ and $b(\alpha,s)$ (described below) to deal with the interaction between different 
strategies.  If a strategy $\alpha$ is initialized, then its components $U_{\alpha}$ and $V_{\alpha}$ 
and its parameters become undefined.  

To see how strategies for different $R$ requirements interact, 
assume that $\alpha \subsetneq \beta$.  If $\alpha*triv \subseteq \beta$, 
then $\beta$ ignores $\alpha$ when it acts.  If $\alpha*fin \subseteq \beta$, then $\beta$ 
assumes that the parameters $u(\alpha,s)$ and $v(\alpha,s)$ have reached their final values 
and $\beta$ makes sure that its chosen components $U_{\beta}$ and $V_{\beta}$ each 
contain more type D trees than $V_{\alpha}$.  (That is, $\beta$ tries to diagonalize using components 
in $T$ which have indices greater than those used by $\alpha$ to diagonalize.)  This restriction causes 
no problems for $\beta$ because if the $\varphi_{\beta}$-tree is isomorphic to $T$ and 
$f_{\beta}$ is a nontrivial self-embedding of the $\varphi_{\beta}$-tree, then $f_{\beta}$ 
must map infinitely many components of the $\varphi_{\beta}$-tree to components which 
contain strictly more type D trees.  

If $\alpha*u_{\infty} \subseteq \beta$, then $\beta$ assumes that the parameter $u(\alpha,s)$ 
will approach infinity in the limit.  In this case, $\beta$ waits to work behind $\alpha$ in the sense 
that $\beta$ only works with components $U_{\beta}$ and $V_{\beta}$ which contain 
strictly fewer type D trees that $T_{u(\alpha,s)}$.  If $\alpha*u_{\infty}$ is on the true path, then 
$u(\alpha,s)$ will go to infinity, so $\alpha$ may occasionally delay $\beta$ from 
acting but will not prevent $\beta$ from succeeding in the end. 

Similar, if $\alpha*v_{\infty} \subseteq \beta$, then $\beta$ assumes that the parameter 
$u(\alpha,s)$ has reached its limit and that $v(\alpha,s)$ will approach infinity.  In this case, 
$\beta$ works in between $U_{\alpha}$ and $V_{\alpha}$ in the sense that $\beta$ works with 
components $U_{\beta}$ and $V_{\beta}$ for which the number of type D trees is strictly greater 
than in $U_{\alpha}$ and is strictly less than in $V_{\alpha}$.  If $\alpha*v_{\infty}$ is on the 
true path, then the parameter $u(\alpha,s)$ eventually reaches a finite value (which $\beta$ 
can work beyond in the sense described above) and $v(\alpha,s)$ does go to infinity, so 
$\alpha$ may cause 
$\beta$ to delay acting occasionally, but will not prevent $\beta$ from succeeding in the end.  

To implement these restrictions, we introduce the parameters $a(\beta,s)$ and $b(\beta,s)$.  
When $\beta$ is first eligible to act (or first eligible to act after having been initialized), 
$a(\beta,s)$ is defined to be large.  (That is, it is defined to be larger than any number used in the 
construction so far.)  The important feature of $a(\beta,s)$ (which we verify after the 
construction) is that it is greater than $u(\alpha,s)$ for all $\alpha$ such that $\alpha*v_{\infty} 
\subseteq \beta$ and it is greater than $v(\alpha,s)$ for all $\alpha$ such that 
$\alpha*fin \subseteq \beta$.  When $\beta$ searches for components $U_{\beta}$ and 
$V_{\beta}$ to use in its diagonalization, it only looks at components in the $\varphi_{\beta}$-tree 
which have strictly more than $a(\beta,s)$ many type D trees.   

At each stage $s$, $\beta$ defines its parameter $b(\beta,s)$ to be the minimum of 
$v(\alpha,s)$ for all $\alpha$ such that $\alpha*v_{\infty} \subseteq \beta$ and 
$u(\alpha,s)$ for all $\alpha$ such that $\alpha*u_{\infty} \subseteq \beta$.  When $\beta$ 
looks for its components $U_{\beta}$ and $V_{\beta}$, it only looks at components in the 
$\varphi_{\beta}$-tree which have strictly less that $b(\beta,s)$ many type D trees.  

Suppose $\beta$ is successful at finding components $U_{\beta}$ and $V_{\beta}$ which satisfy 
these size restrictions.  Because $U_{\beta}$ has more than $a(\beta,s)$ many type D trees, 
$\beta$ is working with components which are beyond (in the sense of the indices of 
potentially isomorphic components in $T$) those used by all $\alpha$ for which 
$\alpha*fin \subseteq \beta$.  Because $V_{\beta}$ has fewer than $b(\beta,s)$ many 
type D components, $\beta$ is working with components which are behind (in the same sense 
as the previous sentence) the components used by each $\alpha$ such that $\alpha*u_{\infty} 
\subseteq \beta$.  Finally, because $U_{\beta}$ has more than $a(\beta,s)$ many type D 
trees and $V_{\beta}$ has fewer than $b(\beta,s)$ many trees, $\beta$ is working between 
the witness components for all $\alpha$ such that $\alpha*v_{\infty} \subseteq \beta$.  Thus, 
these parameters succeed in forcing $\beta$ to work with components of the intuitively 
correct size.  

We now present the formal construction.  At the beginning of stage $s$, $T$ will contain a 
root node plus finite components $T_0, \ldots, T_{s-1}$.  We begin by adding component 
$T_s$ consisting of a single type A component, a single type B component and $s$ many 
type D components.  We proceed to let strategies (beginning with the unique strategy for 
the requirement $R_{\langle 0,0 \rangle}$) act as in the basic module described below.  Once a 
strategy $\alpha$ with $|\alpha| = s$ acts, we end the stage and initialize all strategies of lower 
priority than $\alpha$.  If a strategy is not eligible to act at stage $s$ but is not initialized, then 
its parameters retain their values.  

\textit{Basic module for strategy} $\beta$.  When $\beta$ is first eligible to act (or first eligible to 
act after being initialized), define $a(\beta,s)$ to be large.  This parameter retains its value unless 
$\beta$ is initialized.  At every $\beta$ stage, define 
\[
b(\beta,s) = \min( \{ v(\alpha,s) \mid \alpha*v_{\infty} \subseteq \beta \} \cup 
\{ u(\alpha,s) \mid \alpha*u_{\infty} \subseteq \beta \} )
\]
(we take as a convention that the minimum of the empty set is $\infty$).  If $U_{\beta}$ 
and $V_{\beta}$ were defined at a previous $\beta$ stage (since the last initialization of $\beta$), then define $u(\beta,s)$ to be 
the number of type D components in $U_{\beta}$ and $v(\beta,s)$ to be the number of 
type D components in $V_{\beta}$.  (Throughout this module, we assume that the 
$\varphi_{\beta}$-tree at stage $s$ is a tree of height four and that we have identified its 
root node.  We can identify nodes at level one in the $\varphi_{\beta}$-tree using type B 
trees and we assume that no type D component in the $\varphi_{\beta}$ tree ever grows.  
If any of these conditions are not true, we let $\beta*triv$ act.) 

\textit{Step 1.}  Check if there are disjoint components $U_{\beta}$ and $V_{\beta}$ in the 
$\varphi_{\beta}$-tree such that $U_{\beta}$ has $> a(\beta,s)$ many type D components, 
$V_{\beta}$ has $< b(\beta,s)$ many type D components, $f_{\beta}$ is defined on all of 
$U_{\beta}$ and $f_{\beta}$ is an embedding of $U_{\beta}$ into $V_{\beta}$.  If there are no 
such components, then let $\beta*triv$ act.  

If there are such components, then fix such $U_{\beta}$ and $V_{\beta}$.  (These 
components do not change at future stages unless $\beta$ is initialized.)  Define 
$u(\beta,s)$ and $v(\beta,s)$ as above and \textit{set up to diagonalize}.  First, 
check if $T_{u(\beta,s)}$ has a type A component.  If not, then add such a component to 
$T_{u(\beta,s)}$.  If so, then $T_{u(\beta,s)}$ remains unchanged.  Second, check if 
$T_{v(\beta,s)}$ has a type A component.  If so, then we add an element to this component to 
make it into a type C component.  If not, then $T_{v(\beta,s)}$ remains unchanged.  
Let $\beta*fin$ act.  When $\beta$ is next eligible to act, it acts in Step 2 (unless it is initialized).

\textit{Step 2.}  Let $t$ be the previous $\beta$ stage.  Break into the following three subcases.  Unless 
$\beta$ proceeds to Step 3 (or is initialized), $\beta$ acts in Step 2 again at the next $\beta$ 
stage.
\begin{enumerate}

\item[2(a).]  If $u(\beta,s) > u(\beta,t)$, then let $\beta*u_{\infty}$ act.  

\item[2(b).]  If $u(\beta,s) = u(\beta,t)$ and $v(\beta,s) > v(\beta,t)$, then let $\beta*v_{\infty}$ act.  

\item[2(c).]  If $u(\beta,s) = u(\beta,t)$ and $v(\beta,s) = v(\beta,t)$, then check if $u(\beta,s) < v(\beta,s) < 
b(\beta,s)$.  If not, then let $\beta*fin$ act.  If so, then set up to diagonalize as described in the second paragraph of Step 1.  ($U_{\beta}$ and $V_{\beta}$ 
are already defined, so setting up to diagonalize consists in making sure that the corresponding trees $T_{u(\beta,s)}$ and 
$T_{v(\beta,s)}$ have the correct number of A components.  Also, once we have set up to 
diagonalize with a particular choice of $u(\beta,s)$ and $v(\beta,s)$, future setting up procedures will not add elements to 
$T_{u(\beta,s)}$ and $T_{v(\beta,s)}$ unless these parameters change.)  Check if $U_{\beta} \cong T_{u(\beta,s)}$, 
$V_{\beta} \cong T_{v(\beta,s)}$, $f_{\beta}$ is defined on all of $U_{\beta}$ and gives an embedding 
of $U_{\beta}$ into $V_{\beta}$.  If any of these conditions fail, then 
let $\beta*fin$ act.  If all of these conditions hold, then proceed to Step 3.

\end{enumerate}

\textit{Step 3.}  In this step, $\beta$ diagonalizes as follows.   The single type A components in 
$T_{u(\beta,s)}$ and $U_{\beta}$ are called the designated components of $T_{u(\beta,s)}$ 
and $U_{\beta}$ (respectively) and the image of the designated component in $U_{\beta}$ 
under $f_{\beta}$ is called the designated component of $V_{\beta}$.  If the designated component 
of $V_{\beta}$ is a type B tree, then add an element to the designated component of 
$T_{u(\beta,s)}$ to make it into a type C tree.  If the designated component of $V_{\beta}$ 
is a type C or D tree, then add an element to the designated component of $T_{u(\beta,s)}$ 
to make it into a type B tree.  Let $\beta*fin$ act.  When $\beta$ is next eligible to act, it acts in 
Step 4.

\textit{Step 4.}  Let $t < s$ be the last $\beta$ stage.  Check if $u(\beta,s) = u(\beta,t)$ and 
$v(\beta,s) = v(\beta,t)$.  If so, then let $\beta*fin$ act.  If not, then return to Step 3.  

This completes the formal description of the construction.  As usual, we say that a strategy 
$\alpha$ is on the true path if $\alpha$ is the leftmost strategy of length $|\alpha|$ which is 
eligible to act infinitely often.  Because the tree of strategies is finitely branching and because 
strategies of length up to $s$ get to act at stage $s$, the true path is infinite.  

\begin{lem}
\label{lem:tp}
Let $\beta$ be a strategy on the true path.  
\begin{enumerate}
\item $\beta$ is initialized only finitely often.
\item $a(\beta,s)$ reaches a finite limit $a(\beta)$.  
\item Unless $\beta*triv$ is on the true path, $\beta$ eventually defines the components 
$U_{\beta}$ and $V_{\beta}$ permanently.  Once these components have been permanently 
defined, the parameters $u(\beta,s)$ and $v(\beta,s)$ are always defined and are increasing 
in $s$.  Furthermore, if $\beta*triv$ acts infinitely often, then $\beta*triv$ is on the true path.
\item If $\beta*u_{\infty}$ is on the true path, then $\lim_s u(\beta,s) = \infty$ and 
$U_{\beta}$ has infinitely many type D components.
\item If $\beta*v_{\infty}$ is on the true path, then $\lim_s v(\beta,s) = \infty$ and 
$V_{\beta}$ has infinitely many type D components.
\item If $\beta*fin$ is on the true path, then $\lim_s u(\beta,s) = u(\beta)$ and 
$\lim_s v(\beta,s) = v(\beta)$ both exist, $U_{\beta}$ has $u(\beta)$ many type D components 
(in the limit) and $V_{\beta}$ has $v(\beta)$ many type D components (in the limit).   
\item $\lim_s b(\beta,s) = \infty$.  
\end{enumerate}
\end{lem}

\begin{proof}
The verification of these properties proceeds by induction on $\beta$ and is standard.  For 
Properties 1 and 2, let $s$ be the least $\beta$ stage such that $s \geq |\beta|$ and the path in the 
tree of strategies is never to the left of $\beta$ after stage $s$.  Property 1 follows because 
$\beta$ is never initialized after stage $s$ and Property 2 follows because $\beta$ defines 
$a(\beta,s)$ at stage $s$ and this definition can only be removed by initialization.  

For Property 3, assume that $\beta*triv$ is not on the true path and $\beta$ is not 
initialized after stage $s$.  In this case, $\beta$ must eventually move from Step 1 to Step 2 
in the basic module at some stage after $s$ and $\beta$ defines $U_{\beta}$ and $V_{\beta}$ 
at this stage.  Because $\beta$ is not initialized after this time, the 
only way that $\beta$ could return to taking outcome $\beta*triv$ is if the $\varphi_{\beta}$-tree 
violated the axioms of a tree or had height greater than four or added an extra element to 
some type D component in $U_{\beta}$ or $V_{\beta}$.  In any of these cases, $\beta$ would 
take outcome $\beta*triv$ at all future $\beta$ stages and hence $\beta*triv$ would be on the 
true path.  It is clear that if this situation does not occur, then the parameters $u(\beta,t)$ and 
$v(\beta,t)$ are defined and increasing in $t$ at every subsequent stage.  Furthermore, it follows 
from these comments that if $\beta$ takes outcome $\beta*triv$ infinitely often, then 
$\beta*triv$ is on the true path.  

For Property 4, assume that $\beta*u_{\infty}$ is on the true path and let $t$ be a stage such that $\beta$ is never initialized after $t$ 
and $\beta*triv$ is not eligible to act after $t$.  The strategy $\beta*u_{\infty}$ 
is only eligible to act at $\beta$ stages after $t$ at which the parameter $u(\beta,s)$ has increased.  Because  
this parameter measures the number of type D components in $U_{\beta}$, Property 4 
follows.  The proof of Property 5 is essentially the same.  

For Property 6, assume that $\beta*fin$ is on the true path.  There must be a stage $t$ after which 
$\beta$ is never initialized and none of $\beta*triv$, $\beta*u_{\infty}$ or $\beta*v_{\infty}$ are 
ever eligible to act.  Therefore, after stage $t$, the parameters $u(\beta,s)$ and $v(\beta,s)$ 
never increase and the number of type D components in $U_{\beta}$ and $V_{\beta}$ never 
increase (and none of these type D components grow).  

For Property 7, let $s$ be the least stage such that $\beta$ is never initialized after $s$.  
If there are no strategies $\alpha$ such that $\alpha*u_{\infty} \subseteq \beta$ or 
$\alpha*v_{\infty} \subseteq \beta$, then $b(\beta,t) = \infty$ for all $t \geq s$.  
Otherwise, by the induction hypothesis, for each $\alpha*u_{\infty} \subseteq \beta$, the value of 
$u(\alpha,t)$ approaches infinity and for each $\alpha*v_{\infty} \subseteq \beta$, the 
value of $v(\alpha,t)$ approaches infinity.  Therefore, $b(\beta,t)$ approaches infinity.  
\end{proof}

\begin{lem}
\label{lem:aprop1}
For all strategies $\beta$ and all $\beta$ stages $s$, 
$a(\beta,s)$ is greater than $\max (\{u(\alpha,s) \mid \alpha*v_{\infty} \subseteq \beta \} \cup 
\{ v(\alpha,s) \mid \alpha*fin \subseteq \beta \} )$.  
\end{lem}

\begin{proof}
This lemma follows from three observations.  First, whenever $\beta$ is initialized, it 
defines $a(\beta,s)$ to be large at the next $\beta$ stage.  This large value is by definition greater 
than $\max (\{u(\alpha,s) \mid \alpha*v_{\infty} \subseteq \beta \} \cup 
\{ v(\alpha,s) \mid \alpha*fin \subseteq \beta \} )$.  Second, 
if $\alpha*v_{\infty} \subseteq \beta$ and $u(\alpha,s)$ increases, then $\alpha$ takes outcome 
$u_{\infty}$ and $\beta$ is initialized.  Third, if $\alpha*fin \subseteq \beta$ and 
$v(\alpha,s)$ increases, then $\alpha$ takes outcome $v_{\infty}$ and $\beta$ is 
initialized.  
\end{proof}

\begin{lem}
\label{lem:setup}
Assume that $\beta$ sets up to diagonalize at stage $s$.  No strategy $\alpha \neq \beta$ 
can add elements after this point to $T_{u(\beta,s)}$ (unless $\beta$ is initialized or $u(\beta,s)$ 
increases at a later stage) or to $T_{v(\beta,s)}$ (unless $\beta$ is initialized or 
$v(\beta,s)$ increases at a later stage).  
\end{lem}

\begin{proof}
$\beta$ can set up to diagonalize in either Step 1 or Step 2 of the basic module.  In either case, 
we have $a(\beta,s) < u(\beta,s) < v(\beta,s) < b(\beta,s)$ and $\beta$ takes outcome 
$\beta*fin$ when $\beta$ performs this action.  To show that no $\alpha \neq \beta$ can 
add elements to $T_{u(\beta,s)}$ or $T_{v(\beta,s)}$ after this stage unless $\beta$ is 
initialized or the parameters $u(\beta,s)$ or $v(\beta,s)$ change, we break into cases 
depending on the relative priority of $\alpha$ and $\beta$.  

\begin{description}
\item[Case 1.] $\alpha <_L \beta$.  In this case, $\beta$ is initialized at the end of the stage when $\alpha$ acts.

\item[Case 2.] $\beta*fin <_L \alpha$.  In this case, $\alpha$ is initialized when $\beta$ sets up to diagonalize 
at stage $s$.  Therefore, $a(\alpha,t)$ for $t > s$ is greater than $v(\beta,s)$ and hence $\alpha$ 
works with trees with more type D components than $T_{v(\beta,s)}$ and cannot add new elements 
to either $T_{u(\beta,s)}$ or $T_{v(\beta,s)}$. 

\item[Case 3.] $\beta*fin \subseteq \alpha$.  By Lemma \ref{lem:aprop1}, the value 
of $a(\alpha,s)$ is greater than $v(\beta,s)$.  As above, $\alpha$ works 
with trees with strictly more type D components that $T_{u(\beta,s)}$ and $T_{v(\beta,s)}$.

\item[Case 4.] $\alpha \subsetneq \beta$.  We break this case into four subcases.

\begin{description}
\item[Subcase 4a.] $\alpha*triv \subseteq \beta$.  In this subcase, $\alpha$ does not add any elements to $T$ without 
taking an outcome to the left of $\beta$ and initializing $\beta$. 

\item[Subcase 4b.] $\alpha*fin \subseteq \beta$.  By Lemma \ref{lem:aprop1}, 
$v(\alpha,s) < a(\beta,s)$.  Hence, $\alpha$ works with trees with strictly fewer type D 
components than $T_{u(\beta,s)}$ and $T_{v(\beta,s)}$.  If  either the parameter 
$u(\alpha,s)$ or $v(\alpha,s)$ should increase at a later stage, $\alpha$ would take outcome 
$\alpha*u_{\infty}$ or $\alpha*v_{\infty}$ and $\beta$ would be initialized. 

\item[Subcase 4c.] $\alpha*v_{\infty} \subseteq \beta$.  By Lemma \ref{lem:aprop1} and the definition of $b(\beta,s)$ 
\[
u(\alpha,s) < a(\beta,s) < u(\beta,s) < v(\beta,s) < b(\beta,s) < v(\alpha,s).
\]
Therefore, $U_{\alpha}$ has strictly fewer type D components than 
$T_{u(\beta,s)}$ (unless $u(\alpha,s)$ later increases in which case $\beta$ is initialized) 
and $V_{\alpha}$ has strictly more type D components than $T_{u(\beta,s)}$ and 
$T_{v(\beta,s)}$ (unless these parameters increase at a later stage). 

\item[Subcase 4d.] $\alpha*u_{\infty} \subseteq \beta$. By the definition of $b(\beta,s)$ 
\[
u(\beta,s) < v(\beta,s) < b(\beta,s) < u(\alpha,s).
\]
Therefore, $\alpha$ works with components that contain strictly more type D trees that 
$T_{u(\beta,s)}$ and $T_{v(\beta,s)}$ (unless these parameters increase at a later stage). 

\end{description}
\end{description}

\end{proof}

\begin{lem}
\label{lem:finite}
Each component $T_k$ is finite in the limit.
\end{lem}

\begin{proof}
By the construction, a component $T_k$ can only grow at stage $s > k$ if there is a strategy 
$\beta$ such that $u(\beta,s) = k$ and $\beta$ sets up to diagonalize or diagonalizes 
at stage $s$, or such that $v(\beta,s) = k$ and $\beta$ sets up to diagonalize at stage $s$.  Because 
the parameters $a(\beta,s)$ are always chosen large, only finitely many strategies $\beta$ 
can ever have $u(\beta,s) = k$ or $v(\beta,s) = k$ for any fixed value of $k$.  Therefore, it suffices to show that such 
strategies only cause $T_k$ to grow finitely often.  

Consider the case when $v(\beta,s) = k$ and $\beta$ sets up to diagonalize.  In this 
situation, $\beta$ could cause $T_k$ to grow by adding an element to change a type A 
component in $T_k$ into a type C component.  By Lemma \ref{lem:aprop1}, unless 
$\beta$ is initialized or $v(\beta,s)$ increases at a later stage, no other strategy can add 
an element to $T_k$ after stage $s$.  (Even if $u(\beta,s)$ later increases and $\beta$ 
sets up to diagonalize again with the same value of $v(\beta,s)$, the tree $T_k$ will not grow 
because we have already removed the type A component and no other strategy can have added 
a new type A component.)  

If $\beta$ is initialized after setting up to diagonalize, then it will work with components 
which contain strictly more type D components than $T_k$ in the future.  Therefore, in this case 
$\beta$ causes only finitely much change to $T_k$.  

If the value of $v(\beta,s)$ increases at a later stage, then $\alpha$ might reach a later 
stage $t$ in which $u(\beta,t) = k$ and $\beta$ sets up to diagonalize at stage $t$.  In this 
situation, $\beta$ will change $T_k$ by adding a type A component.  (We could also arrive 
at this situation without having $v(\beta,s) = k$ at some earlier stage in which case $\beta$ 
will only add a type A component if $T_k$ does not already have such a component.)  

By Lemma \ref{lem:aprop1}, unless $\beta$ is initialized or the value of $u(\beta,t)$ increases 
at a later stage, no strategy $\alpha \neq \beta$ can add elements to $T_k$.  If 
$\beta$ is later initialized or if $u(\beta,t)$ increases at a later stage, then $\beta$ will work 
with components with strictly more type D trees than $T_k$ and hence $\beta$ will only cause 
only finitely much growth to $T_k$.  

If $\beta$ later diagonalizes with $u(\beta,t) = k$, then $\beta$ will add an element to $T_k$ to 
change the type A component to either type B or type C.  However, because $\beta$ diagonalizes 
at most once with any given components, it will not add any more elements to $T_k$ at a 
future stage.  Therefore, $\beta$ only adds finitely many elements to $T_k$.  
\end{proof}

\begin{lem}
\label{lem:reqsat}
All requirements $R_{\langle e,i \rangle}$ are satisfied.
\end{lem}

\begin{proof}
Let $\beta$ be the strategy on the true path such that $|\beta| = \langle e,i \rangle$.  Assume for a 
contradiction that the $\varphi_{\beta}$-tree is isomorphic to $T$ and that $f_{\beta}$ is a 
nontrivial self-embedding of the $\varphi_{\beta}$ tree.  

Let $s$ be a stage after which $\beta$ is never initialized and $a(\beta,s)$ has reached its 
final value.  Because $\lim_s b(\beta,s) = \infty$, there must be a $\beta$ stage $t > s$ and 
disjoint components $U$ and $V$ of the $\varphi_{\beta}$-tree such that 
the number of type D components in $U$ is 
strictly greater than $a(\beta,t)$, the number of type D components in $V$ is strictly less than 
$b(\beta,t)$ and $f_{\beta}$ is an embedding of $U$ into $V$.  At this stage, $\beta$ defines 
$U_{\beta}$ and $V_{\beta}$ permanently.  

By the choice of $U_{\beta}$ and $V_{\beta}$, we know that $f_{\beta}$ is an embedding from 
$U_{\beta}$ into $V_{\beta}$ at stage $t$.  These components may gain new type D components 
after stage $t$, 
but because $f_{\beta}$ is a self-embedding of the $\varphi_{\beta}$-tree, it must eventually 
become defined on all of $U_{\beta}$ as it grows and continue to be an embedding into 
$V_{\beta}$.  Furthermore, it is possible that $V_{\beta}$ gains new type D components more 
quickly than the parameter $b(\beta,s)$ grows.  However, since $\lim_s b(\beta,s) = \infty$, 
we must eventually be in the situation where $V_{\beta}$ has stopped growing and 
it has strictly fewer than $b(\beta,s)$ many type D components.

Therefore, we can assume without loss of generality that $t$ is a $\beta$ stage, that $U_{\beta}$ 
and $V_{\beta}$ have completely enumerated all of their type D components by 
stage $t$, that $V_{\beta}$ has 
strictly fewer than $b(\beta,t)$ many type D components and that $f_{\beta}$ is an 
embedding of $U_{\beta}$ into $V_{\beta}$.  Notice that these assumptions imply that the 
parameters $u(\beta,t)$ and $v(\beta,t)$ have reached their limits and that 
$u(\beta,t) < v(\beta,t) < b(\beta,t)$.   

At stage $t$, $\beta$ sets up to diagonalize (in 
either Step 1 or Step 2 depending on the previous actions of $\beta$).  $\beta$ adds 
a type A component to $T_{u(\beta,t)}$ (if necessary) and changes the type A component in 
$T_{v(\beta,t)}$ to a type C component (if necessary).  Because the $\varphi_{\beta}$-tree 
is isomorphic to $T$ and because neither $U_{\beta}$ nor $V_{\beta}$ gains a type D component 
after stage $t$, there must be a stage $t' > t$ at which $U_{\beta} \cong T_{u(\beta,t')}$ and 
$V_{\beta} \cong T_{v(\beta,t')}$.  At this stage, $\beta$ moves to Step 4 of the basic 
module and diagonalizes by changing the designated component in $T_{u(\beta,t')}$ so that 
it cannot be embedded into the designated component of $V_{\beta}$.  

Because $U_{\beta} \cong T_{u(\beta,t')}$ before this additional element is added, we know that the 
number of each type of component in $U_{\beta}$ and $T_{u(\beta,t')}$ (before the additional 
element is added) match up.  Furthermore, because $\beta$ is never initialized again and both 
$u(\beta,s)$ and $v(\beta,s)$ have reached their limits, by Lemma \ref{lem:aprop1} we know 
that $T_{u(\beta,t')}$ does not change again after stage $t'$.  Therefore, to make 
$U_{\beta} \cong T_{u(\beta,t')}$ (after the extra element is added), $U_{\beta}$ must change 
its designated component to match the new type of the designated component in 
$T_{u(\beta,t')}$.  Furthermore, because $U_{\beta}$ 
is already committed by $f_{\beta}$ to embedding the designated component of $U_{\beta}$ 
into the designated component of $V_{\beta}$ and because the new type of the designated 
component of $T_{u(\beta,t')}$ does not embed into the designated component of $V_{\beta}$, 
the embedding $f_{\beta}$ cannot be extended in a way that is compatible with extending 
the designated component of $U_{\beta}$.  This fact gives the desired contradiction.  
\end{proof}

This completes the verification that our construction succeeds.

\section{Type 2 trees}
\label{sec:type2}

Recall that a type 2 computable tree is one which has no maximal $\omega$-nodes and which has an isolated path.  (The type 2 trees we construct 
below will all be finitely branching and hence will have no $\omega$-nodes at all.)  It is well known that there are such trees $T$ 
which have a single path and such that this path codes $0'$.  (The successor relation on $T$ can even be computable.  Among other places, such a construction is 
contained in the proof that K\"{o}nig's Lemma for finitely branching trees is equivalent to $\text{ACA}_0$ in Simpson \cite{simp:book}.)  We 
construct such a tree and show that any nontrivial self-embedding of it can compute $0'$.    

\begin{lem}
\label{tech}
If $T$ is a tree ordering on $\om$ such that

\begin{enumerate}
\item[a)]  $T$ is finitely branching

\item[b)] $\A m, n  ( m\succeq n \implies  m\geq n)$

\item[c)] $\A n  ( n\succeq \text{the immediate predecessor of $n+1$})$
\end{enumerate}
\noi then $T$ has exactly one path $X$. Furthermore, if this path $X$ is written as $\lambda = x_0 \prec x_1 \prec x_2 \prec \cdots$ where 
$x_{i+1}$ is the successor of $x_i$, then $x_n = \max \{m: {\rm ht}(m)\leq n\}$.   
\end{lem}

\begin{proof}
Let $x_n=\max\{m: {\rm ht}(m)\leq n\}$ which exists by a).  For any fixed $n$, we will show by induction  that  $l\succeq x_n$ for all $l\geq x_n$. 
Thus for any $n$, $x_n$ is an infinite node.  By b), $x_n$ has height $n$ and is the only infinite node of height $n$.  
Therefore $x_0 \prec x_1 \prec x_2 \prec \cdots$ gives the unique path through $T$. 

Fixing $n \in \om$, we show that $l \succeq x_n$ for all $l \geq x_n$ by induction on $l$.  
For the base of the induction, we have trivially that $x_n\succeq x_n$. For the induction step, suppose that 
$l\geq x_n$ and $l\succeq x_n$ 
and we show that $l+1\succeq x_n$.  Let $p$ be the immediate  predecessor of $l+1$ in $T$.
 By c), $l\succeq p$ and hence $p$ is comparable to $x_n$. So either $p\succeq x_n$, in which case $l+1\succeq x_n$ and we are done, or 
 $p\prec x_n$, in which case ${\rm ht}(l+1)\leq n$ (as ${\rm ht}(x_n)\leq n$) contradicting the fact that 
 $x_n=\max\{m:{\rm ht}(m)\leq n\}$ as $l+1>l\geq x_n.$ 
\end{proof}

\begin{thm}
\label{iso}
There is a computable type 2 tree $T$ such that any nontrivial self-embedding of $T$ computes $0'$.
\end{thm}

\begin{proof}
We  construct $T$ to have exactly one infinite path $X$ such that ${\rm deg}_T(X)\geq 0'$.  Fix a c.e.~set $K$ of degree $0'$.  

$\la T, \preceq\ra$ is built computably in stages denoted $\la T_s, \preceq_s\ra $ with 
$T_s= \{0, 1, 2, \dots s\}$ and $\preceq_s=\preceq\cap \bigl( T_s\times T_s\bigr)$ for all $s$.  
$T = \bigcup_s T_s = \om$ and $\preceq = \bigcup_s \preceq_s$.   
At each stage $s$ we  designate an element $n_s$ of $T_s$ to be the immediate predecessor of $s+1$ and $\preceq_{s+1}$ is 
defined to be the reflexive and transitive closure of $\preceq_s \cup \{(n_s, s+1)\}$. 

Let $\preceq_0=\{(0,0)\}$ and at stage $s$ let 
\[
n_s=
\begin{cases} 
\max\{m \preceq s \mid K_{s+1}[{\rm ht}(m)]=K_s[{\rm ht}(m)]\} & \text{ if this set is nonempty,} \\
0& \text{ otherwise}.
\end{cases}
\]

\begin{lem}
$T$ has exactly one infinite path $X$ and this path computes $K$. 
\end{lem}

\begin{proof}
Conditions b) and c) of Lemma \ref{tech} are satisfied immediately by the construction of $T$.  Furthermore, because new nodes are always 
added as (current) leaves in $T$, the height of any node is fixed once it is placed in $T$.  Therefore, we can speak of 
$\text{ht}(x)$ for any $x \in T$ without reference to a stage number.  

To see that a) is satisfied, we reason by contradiction.  Assume that $x \in T$ is an $\omega$-node.  It follows that there are infinitely many 
stages $s_0 < s_1 < \cdots$ such that $n_{s_i} = x$.  At stage $s_0$, $x$ gains $s_0+1$ as a successor and hence $x$ is no longer a leaf 
in $T$ after stage $s_0$.  Consider any stage $s_i$ for $i \geq 1$.  Since $x = n_{s_i}$, we have $x \preceq s_i$ and so $x \prec s_i$ 
(because $s_i$ was added as a leaf at the previous stage and we know $x$ is no longer a leaf in $T$).  Therefore, $x$ has a successor 
$y_i$ such that $x \prec y_i \preceq s_i$.  By the definition of $n_{s_i} = x$, we know $K_{s_i+1}[\text{ht}(x)] = K_{s_i}[\text{ht}(x)]$ but 
$K_{s_i+1}[\text{ht}(y_i)] \neq K_{s_i}[\text{ht}(y_i)]$.  However, $\text{ht}(y_i) = \text{ht}(x)+1$, so $K_{s_i+1}[\text{ht}(x)] = K_{s_i}[\text{ht}(x)]$ and 
$K_{s_i+1}[\text{ht}(x)+1] \neq K_{s_i}[\text{ht}(x)+1]$ for all $i \geq 1$.  Because $K$ is a c.e.~set, there can be at most one such stage $s_i$, 
giving the desired contradiction.  (This argument really shows that $T$ is binary branching.)  Therefore a) holds and $T$ has exactly one path $X$.  

We next show that $X \geq_T K$.  Because $s+1$ is added as a (current) leaf of $T_s$ at stage $s$, the successor relation 
on $T$ is computable.  Therefore, from the set $X$, we can compute the sequence $x_0 \prec x_1 \prec \cdots$ such that $x_0$ is the root of 
$T$, $x_{i+1}$ is the successor of $x_i$ and $X = \{ x_i \mid i \in \omega \}$.  

By Lemma \ref{tech} we know that $x_n = \max\{ m \mid \text{ht}(m) \leq n \}$.  Furthermore, by the proof of Lemma \ref{tech}, we know 
that for all $s \geq x_n$, we have $x_n \preceq s$.  We claim that for all $n$, $K_{x_n}[n] = K[n]$.  Suppose not and fix $n$ such that $K_{x_n}[n] \neq K[n]$.  
Let $s \geq x_n$ be such that $K_{s+1}[n] \neq K_s[n]$.  Because $x_n \preceq s$ and $\text{ht}(x_n) = n$, we have that 
$n_s \prec x_n$.  Therefore, $x_n \not \preceq s+1$ contradicting the fact that $x_n \preceq s+1$.  
\end{proof}

\begin{lem}
Any nontrivial self-embedding $\varphi$ of $T$ computes  $K$.
\end{lem}

\begin{proof}
Let $\varphi$ be a nontrivial self-embedding of $T$ and let $m$ be some node on $X$ such that $\varphi(m)\neq m$. (That such a node exists is a consequence 
of $X$ being the only path and there being no nontrivial self-embeddings of finite trees.)  $\varphi(m)$ must also lie on $X$  as $T$ has only one infinite path.  
By induction one sees that  for all $n$, $\varphi^n(m)\succeq x_n$ and $\varphi^n(m) \in X$.  Therefore, for all $n$, $\varphi^n(m) \geq x_n$, 
$K_{\varphi^n(m)}[n]=K_{x_n}[n]=K[n]$ and $\varphi\geq_T K$.
\end{proof}

This completes the proof of Theorem \ref{iso}.
\end{proof}

The result can be improved slightly by replacing  {\em nontrivial} with {\em weakly nontrivial} in \ref{iso}. 
To show this we construct $T'$ from $T$. $T'$ will have, like $T$, a single isolated path $X'$ that computes $K$. $T'$ will be modified however to 
ensure that any weakly nontrivial self-embedding must move a node on $X'$. 

We say a node $n$ on $T$ is {\em just off $X$} if $n$ is not on $X$ but the immediate predecessor is on $X$. Any weakly nontrivial self-embedding 
of $T$ that fixes every node on $X$ must be weakly nontrivial on some finite tree $T(n)$ with $n$ just off $X$. Any finite tree can be properly 
extended to a finite tree that has no weakly nontrivial self embedding by adding nodes extending the leaves so that no two  leaves have the 
same height. This is what we do to ensure that $T(n)$ has no nontrivial self embedding. Extensions may be added at different stages in the 
construction but we ensure that each leaf is extended only a finite amount. 

We first repeat the construction of $T$ using only the even numbers - adding node $2(s+1)$ as the immediate successor to $2n_s$ at stage $s$. 
We describe the placement of every odd number on $T'$.

We begin with $T'_0=T_0$. As before at stage $s$ we determine $2n_s$ and place $2(s+1)$ as its immediate successor. At stage $s$ we also 
find all leaves extending $2n_s$ except $2(s+1)$ and we properly extend all such leaves with successive 
odd numbers so that any two distinct leaves have different heights.  That is, we are guessing that the path $X$ will pass through $2(s+1)$ and 
so all the other successors of $2n_s$ are just off $X$.  Therefore, we want to extend the leaves above these other successors to have different lengths.  
$T'_{s+1}$ is this extension of $T_{s+1}$. We need only show now that every leaf on $T$ is extended only finitely and that all the odd numbers are used.

As in Lemma \ref{tech}, $x_n'$ (the $n^{\text{th}}$ element in the unique path $X'$ in $T'$) will be the numerically greatest even number of height less 
than or equal to $n$. $x_n'$ is added to $T'$ at stage $s=x_n'/2$ and for all $t\geq s$, $2n_t\succeq x_n'$. So no more extensions will be added to the 
leaves above any node just off $X$ whose height is less than or equal to $n$.  Therefore, each leaf is extended finitely only a finite number of times. 

To see that all the odd numbers are used we merely need to note that there are infinitely many nodes just off $X'$ (otherwise $X'$ would be 
computable)  and that we have decreed that each leaf extending  such a node must be properly extended.                                                                                                                       

We next turn our attention to nontrivial self-embeddings in computable trees $S \cong T$ where $T$ 
is a computable type 2 tree.    

\begin{thm}
There is a computable finitely branching tree $S$ with exactly one infinite path (so $S$ is a type 2 tree) such that no computable tree 
classically isomorphic to $S$ has a computable nontrivial self-embedding.
\end{thm}

\begin{proof} 
Let $\la T, \preceq_T\ra$ be the tree constructed in Theorem \ref{isomaxinf}, let $T_i$ (for $i \in \omega$) be the sequence of components of $T$, 
let $\lambda_T$ denote the root of $T$ and let $\lambda_i$ denote the root of the component $T_i$.  Let $A=\{ a_i: i\in \omega \}$ be a set of distinct 
elements disjoint from $T$. We define the tree $\la S, \preceq_S \ra$ as follows:
\begin{enumerate}
\item $S= (T\smallsetminus \{\lambda_T\})\cup A$,
\item $a_0$ is the root of $S$,
\item $\A x,y \in T \setminus \{ \lambda_T \}   \, ( x\preceq_T y \leftrightarrow x\preceq_S y)$, and 
\item $\A i\in \N \, ( a_{i+1}$ and $\lambda_i$ are  immediate successors of $a_i$).

\end{enumerate}

It is straightforward to see that there is a unique tree $S$ satisfying 1, 2, 3 and 4 and that $S$ has exactly the one infinite path given 
by $a_0 \prec a_1 \prec a_2 \cdots$. 

We claim that any nontrivial self-embedding $\delta$ of $S$ computes a nontrivial self embedding of $T$, and hence by Theorem \ref{isomaxinf}, 
$\delta$ must be noncomputable.  To prove this claim, notice that because $a_0 \prec a_1 \prec \cdots$ is the unique path in $S$, we have that for each 
$i \in \omega$, 
$\delta(a_i) = a_j$ for some $j \geq i$.  Because $\delta$ is nontrivial and every subtree off the unique path is finite, we can fix the least $k$ such that 
$\delta(a_k) \neq a_k$.  For every $i \geq k$, $\delta(a_i) = a_j$ for some $j > i$.  In other words, $\delta$ gives an embedding from $T_i$ into $T_j$.  It 
follows that $\delta$ induces a nontrivial self-embedding  $\delta'$ of $T$ given by

\begin{enumerate}
\item $\delta'(\lambda_T)=\lambda_T$ and
\item for all $x\in T\smallsetminus\{\lambda_T\}\ \, (\delta'(x)=\delta(x))$.
\end{enumerate}

To finish the proof, let $\la U, \preceq_U\ra$ be any computable tree classically isomorphic to $S$. We show that any nontrivial self-embedding $\delta$ of 
$U$ computes a nontrivial self-embedding of a computable tree classically isomorphic to $T$. Hence by Theorem \ref{isomaxinf}, $\delta>_T 0$.

As $U$ is isomorphic to $S$, it too has exactly one infinite path $B$. We claim that $B$ is computable.  It is c.e.~because $x\in B$ if and only if there exists 
a chain of height 4 above $x$ (recall all the components of $T$ have height 3). And it is co-c.e.~because $x\nin B$ if and only if there exists a $b\in B$ 
such that $b$ and  $x$ are incomparable.

We build a computable tree $\la V, \preceq_V \ra$ isomorphic to $T$ as follows. Let $V=(U\smallsetminus B)\cup \{\rho\}$, where $\rho$ is a new element which 
will serve as the root of $V$.  For all $x,y\in V$ define $x\preceq_V y$ if and only if $x=\rho$ or $x\preceq_U y$.  
As $B$ is computable, so is $\la V, \preceq_V \ra$, and it is easy to see that  $V$ is isomorphic to $T$. 

From $\delta$ we compute  a nontrivial self-embedding $\delta'$ of $V$. Let $\delta'(\rho)=\rho$ and for all $x\in U\smallsetminus B$, 
let $\delta'(x)=\delta(x)$. $\delta'$ is clearly computable from $\delta$ and hence $\delta>_T 0$. \end{proof}

\section{Type 3 trees}
\label{sec:type3}

Recall that a type 3 computable tree is an infinite computable tree which has no maximal infinite node and no isolated paths.  In 
particular, any infinite computable binary branching tree which has no isolated paths is a type 3 tree.  In Theorem \ref{thm:ntse}, we 
proved that each computable type 3 tree has a nontrivial self-embedding computable in $0''$.  The next theorem shows that this 
bound is optimal.

\begin{thm}
\label{thm:doublecode}
There is an infinite computable binary branching tree $S$ with no isolated paths 
such that any nontrivial self-embedding computes $0''$.  
\end{thm}

Before proving Theorem \ref{thm:doublecode}, we outline the main steps of the proof.  
We begin by giving a particular computable approximation to $0''$ which is conducive to 
our coding methods.  Next, we define a  c.e.~subtree 
$T \subseteq 2^{< \omega}$ such that $0''$ is coded into the branching levels of $T$.  
(A c.e.~subtree $T \subseteq 2^{< \omega}$ is a c.e.~set $T$ of elements of $2^{< \omega}$ which is closed under initial segments.  
The tree order is given by the initial segment relation $\subs$.)  
We say $n$ is a \textit{branching level} of $T$ if there is a string $\sigma \in T$ such  
that $|\sigma| = n$ and both $T(\sigma\ast0)$ and $T(\sigma\ast1)$ are infinite.  We use 
a c.e.~subtree $T$ of $2^{< \omega}$ because it makes the notation easier when verifying 
properties such as where the branching levels occur in $T$ and the fact that $T$ has no 
isolated paths.  

We show that from any nontrivial self-embedding of $T$ we can compute a function dominating 
the branching levels and that any such function computes $0''$.  Finally, we show how to 
define a computable tree $S \cong T$ for which the successor relation is computable.  
Because the branching levels of $T$ are invariant under isomorphisms, we have 
$0''$ coded into the branching levels of $S$.  From any nontrivial self-embedding of 
$S$, we can decode $0''$ as long as we can determine the height of each node in 
$S$.  However, since the successor relation is computable in $S$, we can effectively 
determine the height of any node.  

We begin by developing our computable approximation to $0''$.  
Fix a uniformly c.e.~sequence of c.e.~sets $A_n$ for $n \in \omega$ such that 
$\{ n \, | \, A_n \, \text{is finite} \, \} \equiv_T 0''$.  (For example, we could use the standard enumeration of all c.e.~sets.)  
Without loss of generality, we assume that in the 
uniform enumeration of the $A_n$ sequence, exactly one set gets an element at each stage.  Let 
$f(n) =$ the least $s$ such that the sets among $A_0, \ldots, A_n$ which are finite have 
been completely enumerated by stage $s$.  

\begin{lem}
\label{lem:fcomp}
For any function $b$ which dominates $f$, $0'' \leq_T b \oplus 0'$.
\end{lem}

\begin{proof}
Let $k \in \omega$ be such that $f(x) \leq b(x)$ for all $x \geq k$.  To determine whether $n \in 0''$ for $n \geq k$, ask $0'$ whether 
$A_n$ gets an element after stage $b(n)$.  The answer to this question is no if and only if $n \in 0''$.  
\end{proof}

We want to define a computable approximation $f(n,s)$ to the function $f(n)$ so that 
$f(n) = \liminf_{s} f(n,s)$.  To define $f(n,s)$, 
proceed as follows.  If the sets $A_0, \ldots, A_n$ are all empty at stage $s$, then 
set $f(n,s) = 0$.  If at least one of these sets is nonempty but none of them receives a new 
element at stage $s$, then let $f(n,s) = t$ where $t < s$ is the last stage at which one of 
these sets received an element.  

If we are not in one of these two cases, then at stage $s$, exactly one set among 
$A_0, \ldots, A_n$ gets a new element.  
Let $i_{n,s} \leq n$ be such that $A_{i_{n,s}}$ gets a new element at stage $s$ and let 
$t_{n,s} < s$ be the last stage at which $A_{i_{n,s}}$ received an element.  
(If $s$ is the first stage at which $A_{i_{n,s}}$ gets an element, then set $t_{n,s} = 0$.)  Let 
\[
I_{n,s} = \{ j \leq n \, | \, A_j \, \text{has received an element since} \, t_{n,s} \}.
\]
$I_{n,s}$ represents our current guess at which sets among $A_0, \ldots, A_n$ are infinite.  
Let $f(n,s) = t$ where $t < s$ is the greatest stage such that there exists a $j \leq n$ for which 
$j \not \in I_{n,s}$ and $A_j$ gets an element at stage $t$.  (If the sets $A_j$ for $j \not \in 
I_{n,s}$ are all empty or if $I_{n,s} = \{0,1, \ldots,n \}$, then set $f(n,s) = 0$.)  That is, to calculate $f(n,s)$ 
we look at the sets $A_j$ for $j \leq n$ which we currently think are not infinite and take the last 
stage at which one of these sets received a new element.  The function $f(n,s)$ is  
a total computable function.  

\begin{lem}
\label{lem:fprop}
The function $f(n,s)$ satisfies the following properties.
\begin{enumerate}
\item $f(n) = \liminf_{s} f(n,s)$.
\item For every $k > f(n)$, there is a stage $s_k$ such that for all $t \geq s_k$ either 
$f(n,t) = f(n)$ or $f(n,t) > k$. 
\end{enumerate}
\end{lem}

\begin{proof}
Fix $n$ and break into two cases.  If $A_0, \ldots, A_n$ are all finite, then let 
$u$ be the last stage at which any of these sets gets an element.  Because $f(n,s) = u$ for all 
$s > u$, we have both Property 1 and 2 in this case.  

Otherwise, there is at least one set among $A_0, \ldots, A_n$ which is infinite.  
Let $I$ be the set of all $i \leq n$ such that $A_i$ is infinite and let $u_0 = f(n)$ be the last 
stage such that some $A_j$ with $j \leq n$ and $j \not \in I$ receives an 
element.  Let $u_1 > u_0$ be a stage such that each $A_i$ with $i \in I$ has received 
at least one element between stages $u_0$ and $u_1$.    

Consider any stage $s > u_1$ and split into two cases.  First, if none of the sets 
$A_i$ for $i \leq n$ receives an element at 
stage $s$, then $f(n,s) > u_0$ since $I \neq \emptyset$ and each $A_i$ for $i \in I$ 
received an element after stage $u_0$.  
Second, if one of the $A_i$ sets for $i \in I$ does receive an element at stage 
$s$, then $t_{n,s} > u_0$ since each such set receives an element between stages $u_0$ and 
$u_1$.  Furthermore, $I_{n,s} \subseteq I$ since none of the sets $A_j$ for $j \not \in I$ 
receives an element after stage $u_0$.  

If $I_{n,s} = I$, then $f(n,s) = u_0 = f(n)$.  If $I_{n,s} \subsetneq I$, then $f(n,s) > u_0$ 
since there is an $A_i$ for which $i \in I \setminus I_{n,s}$ and this $A_i$ received an 
element after stage $u_0$.  Therefore, for all $s > u_1$, $f(n,s) \geq u_0 = f(n)$.  

Define a sequence of stages $v_0 < v_1 < v_2 < \cdots$ such that $u_1 < v_0$ and 
at each stage $v_k$, $I_{n,v_k} = I$.  To see that such a sequence exists, consider any stage $t > u_1$.  
We claim there is a stage $s > t$ such that $I_{n,s} = I$.  To find $s$, let $j_n$ be the index such that 
$j_n \leq n$ and $A_{j_n}$ is the last set among $A_i$ with $i \in I$ to receive a new element after 
stage $t$.  Let $s > t$ be the first stage at which $A_{j_n}$ receives a new element.  Since 
$A_{j_n}$ receives a new element at stage $s$, we have $i_{n,s} = j_n$.  As $s > t$ is the first 
stage at which $A_{j_n}$ receives a new element and since $A_{j_n}$ has received an element 
since $u_0$, we have $u_0 < t_{n,s} \leq t$.  Each set $A_i$ with $i \in I$ has now received a new element 
since stage $t$ (and hence since $t_{n,s}$), so $I_{n,s} = I$ as required.  Therefore, we have established the 
existence of the sequence $v_0 < v_1 < \cdots$ with $I_{n,v_k} = I$.    
Since $I_{n,v_k} = I$, we have $f(n,v_k) = u_0 = f(0)$.  Therefore, we have established 
Property 1.  

On the other hand, we can extend our sequence of stages $u_0 < u_1 < u_2 < \cdots$ 
so that each $A_i$, $i \in I$, receives an element between stages $u_k$ and $u_{k+1}$.  
Consider any $s > u_{k+1}$.  If none of the sets $A_i$, $i \in I$, receive an element at 
$s$, then $f(n,s) > u_k$ since $I \neq \emptyset$ and each $A_i$, $i \in I$, has received an 
element since $u_k$.  If some $A_i$, $i \in I$, does receive an element at stage $s$, then either 
$I_{n,s} = I$ (in which case $f(n,s) = u_0 = f(n)$) or $I_{n,s} \subsetneq I$ (in which case 
$f(n,s) > u_k$ since $A_{i_{n,s}}$ has received an element since stage $u_k$).  Therefore, we 
have established Property 2. 
\end{proof}

We define a computable function $g(n,s)$ from $f(n,s)$ that has one further property.  
We define $g(n,s)$ by induction on $s$, and for each $s$ by induction on $n$.  For every $s$, let $g(0,s) = f(0,s)$.  Assume 
$g(i,t)$ has been defined for all $i \leq n$ and $t\leq s$, and we explain how $g(n+1,s)$ is defined.  Let $k_{n,s}$ be the number of stages $t < s$ for which 
$g(i,t) = g(i,s)$ for all $i \leq n$ and let $m_{n,s} =$ the maximum value of $g(i,s)$ for 
$i \leq n$.  Let $l_{n,s} = k_{n,s} + m_{n,s}$.  Define $g(n+1,s) = f(n+1,l_{n,s})$.  

\begin{lem}
\label{lem:gprop}
The function $g(n,s)$ satisfies the following properties.
\begin{enumerate}
\item $f(n) = \liminf_{s} g(n,s)$.
\item For every $k > f(n)$, there is a stage $s_k$ such that for all $t \geq s_k$ either 
$g(n,t) = f(n)$ or $g(n,t) > k$.
\item For every $n$, there are infinitely many stages $s$ at which $g(i,s) = f(i)$ 
for all $i \leq n$.   
\end{enumerate}
\end{lem}

\begin{proof}
We proceed by induction on $n$.  By Lemma \ref{lem:fprop}, these properties hold for 
$n=0$.  Assume these properties hold for $i \leq n$ and we prove them for $g(n+1,s)$.  
Applying Property 3 to $n$, let $u_0 < u_1 < \cdots$ list all the stages at which 
$g(i,s) = f(i)$ for all $i \leq n$.  Let $M =$ the maximum of $f(i)$ for $i \leq n$.  (Because $f$ is nondecreasing, $M$ is really 
just equal to $f(n)$.)  At each stage 
$u_k$, we have $m_{n,u_k} = M$ and $k_{n,u_k} = k$, so by definition $g(n+1,u_k) = f(n+1,M + k)$.  
Therefore, as $k \rightarrow \infty$, $g(n+1,u_k)$ takes on all the values of $f(n+1,t)$ 
for $t > M$.  

Let $t > M$ be a stage for which $f(n+1,t) = f(n+1)$.  (By Property 2 of Lemma \ref{lem:fprop} 
there are infinitely many such stages.)  Let $k = t-M$.  At stage $u_k$, we have 
$g(i,u_k) = f(i)$ for all $i \leq n$ by definition of $u_k$ and we have 
$g(n+1,u_k) = f(n+1,M+k) = f(n+1,t) = f(n+1)$.  Therefore, Property 3 of this lemma 
holds for $n+1$.  

For any $a \in \omega$, let $s_a > u_a$ be a stage such that for every $s > s_a$ and 
every $i \leq n$, either $g(i,s) = f(i)$ or $g(i,s) > a$.  (The existence of $s_a$ follows from 
Property 2 of this lemma applied inductively to $i \leq n$.)  Consider any $s > s_a$.  
By definition, $l_{n,s} = k_{n,s} + m_{n,s}$.  We claim that $l_{n,s} \geq a$.  
There are two cases to consider.  First, suppose 
$g(i,s) = f(i)$ for all $i \leq n$.  In this case, $m_{n,s} = M$ and because $s_a > u_a$, there 
have been at least $a$ many stages $t < s$ for which $g(i,t) = g(i,s) = f(i)$ for all 
$i \leq n$.  Therefore, $k_{n,s} \geq a$, so $l_{n,s} \geq a$.  Second, suppose that for some 
$i \leq n$ we have $g(i,s) \neq f(i)$.  By the choice of $s_a$, $g(i,s) > a$, so $m_{n,s} > a$ and 
$l_{n,s} > a$.  Therefore, in either case $l_{n,s} \geq a$ and so $g(n+1,s) = f(n+1,t)$ for 
some $t \geq a$.  

The previous paragraph established that for all $a$, there is a stage $s_a$ such that for all 
$s > s_a$, there is a $t \geq a$ for which $g(n+1,s) = f(n+1,t)$.  Combining this fact with 
Property 2 of Lemma \ref{lem:fprop} and with the fact that $g(n+1,u_k) = f(n+1)$ for infinitely many 
$u_k$ yields Properties 1 and 2 of this lemma.
\end{proof}

We now put together the last two pieces of our approximating function.  Let $h(n) =$ 
the least stage $s$ for which $K_s[n+1] = K[n+1]$.  Because $h(n)$ is a $\Delta^0_2$ function, it has a computable approximation 
$h(n,s)$ such that $\lim_s h(n,s) = h(n)$.  Finally, let $a(n,s)$ be the computable 
function defined by $a(n,s) = \max \{ g(n,s), h(n,s) \}$.    

\begin{lem}
\label{lem:aprop}
The computable function $a(n,s)$ satisfies the following properties.
\begin{enumerate}
\item $a(n) = \liminf_s a(n,s)$ exists and for all $n$, $a(n) \geq f(n), h(n)$.  
\item For all $n$ and for every $k > a(n)$, there is a stage $s_k$ such that for all $t \geq s_k$, 
either $a(n,t) = a(n)$ or $a(n,t) > k$.
\item For every $n$, there are infinitely many stages $s$ at which $a(i,s) = a(i)$ for all $i \leq n$.
\item For any function $b$ which dominates $a$, $0'' \leq_T b$.  
\end{enumerate}
\end{lem}

\begin{proof}
Properties 1 through 3 follow from Lemma \ref{lem:gprop} and the fact that $h(n) = \lim_s h(n,s)$.  
Property 4 follows from the fact that if $b$ dominates $a$, then $b$ dominates 
both $h$ and $f$.  The fact that $b$ dominates $h$ gives $0' \leq_T b$.  Combining 
this fact with Lemma \ref{lem:fcomp} gives $0'' \leq_T b$.
\end{proof}

We next define a c.e.~subtree $T \subseteq 2^{< \omega}$ such that the branching levels of 
$T$ dominate the function $a(n)$.  The branching levels are the levels that contain branching nodes.  That is, level $k$ in $T$ is a branching 
level if there is a node $\sigma \in T$ such that $|\sigma| = k$ and both $T(\sigma*0)$ and $T(\sigma*1)$ are infinite.

The basic idea of our construction is as follows.  We make the node $0^{a(0)} \in 2^{< \omega}$ the branching node of least 
length in $T$.  Therefore, we need to insure that both $0^{a(0)}*0$ and $0^{a(0)}*1$ have infinitely many extensions in $T$.  
Since we want the next branching level to be above $a(1)$, we make the next branching nodes equal to 
$0^{a(0)}*0*0^{a(1)}$ and $0^{a(0)}*1*0^{a(1)}$.  To do this, we need to ensure that for each node $\sigma$ of length 2, we have 
that the nodes $0^{a(0)}*\sigma(0)*0^{a(1)}*\sigma(1)$ have infinitely many extensions in $T$.  (Below, we will denote these nodes by 
$\tau_2^{\sigma}$.)  We repeat this process by making the next branching nodes have the form 
$0^{a(0)}*\sigma(0)*0^{a(1)}*\sigma(1)*0^{a(2)}$ for all $|\sigma| = 2$.  In other words, for all strings $\sigma$ of length 3, we need 
to ensure that the nodes $0^{a(0)}*\sigma(0)*0^{a(1)}*\sigma(1)*0^{a(2)}*\sigma(2)$ have infinitely many extensions in $T$.  

By repeating the process in the previous paragraph, the branching levels of $T$ will occur at levels of the form 
$n + \sum_{i=0}^n a(i)$ for $n \in \omega$.  We use the approximation $a(n,s)$ to define a c.e.~tree with these branching 
levels and we use Properties 2 and 3 of Lemma \ref{lem:aprop} to verify that these levels are branching levels and that 
no other levels are branching.  Since $a(n) \leq n + \sum_{i=0}^n a(i)$, we have the required domination property.  We then 
need to show how to extract information about the branching levels from any nontrivial self-embedding of our tree.  

We begin with some notation.  For any $s$, any $n \leq s$ 
and any string $\sigma \in 2^{< \omega}$ with $|\sigma| = n+1$, we define 
\[
\tau_{n,s}^{\sigma} = 0^{a(0,s)} \ast\sigma(0)\ast 0^{a(1,s)}\ast\sigma(1)\ast\cdots \ast
0^{a(n,s)}\ast\sigma(n).
\]
For any $n$ and any string $\sigma \in 2^{< \omega}$ such that $|\sigma| = n+1$, we let 
\[
\tau_n^{\sigma} = 0^{a(0)} \ast \sigma(0) \ast 0^{a(1)} \ast \sigma(1) \ast \cdots \ast 
0^{a(n)} \ast \sigma(n).
\]
For any nonempty string $\alpha \in 2^{< \omega}$, let $\alpha'$ denote the string obtained by 
removing the last element of $\alpha$.  Notice that 
\[
(\tau_{n,s}^{\sigma})' = 0^{a(0,s)} \ast \sigma(0) \ast 0^{a(1,s)} \ast \sigma(1) \ast \cdots 
\ast 0^{a(n,s)}
\]
is a string of length $n+\sum_{i = 0}^{n} a(i,s)$ and that $(\tau_n^{\sigma})'$ is a string of 
length $n+\sum_{i=0}^{n} a(i) $.  

As described above, the goal of our construction is to make each node of the form $(\tau_n^{\sigma})'$ a 
branching node of  our c.e.~tree $T$ and we accomplish this goal by making each node $\tau_n^{\sigma}$ have infinitely many extensions 
on $T$.  Because we cannot effectively know which nodes are of the form $\tau_n^{\sigma}$, we have to use the 
approximations $\tau_{n,s}^{\sigma}$.  At stage $s$, we add at least one new node extending each string of the form 
$\tau_{s,s}^{\sigma}$.  We then verify that in the limit, this process makes the branching nodes of $T$ exactly those nodes 
of the form $(\tau_n^{\sigma})'$.  

Once we have such a tree $T$, we show that from any nontrivial embedding $\delta$ of $T$, we can 
effectively obtain a nontrivial embedding $\iota$ of $T$ and a node $\alpha$ such that there are at 
least $n$ many branching levels below $|\iota^{n+1}(\alpha)|$.  By our calculation of the branching 
levels (described in the previous paragraph), the function $c(n) = |\iota^{n+1}(\alpha)|$ dominates the function 
$a(n)$, and hence by Lemma \ref{lem:aprop}, $c$ computes $0''$.  Because $c$ is obtained effectively 
from $\delta$, we conclude that $0'' \leq_T \delta$.  

The subtree $T \subseteq 2^{< \omega}$ is enumerated in stages as a sequence of finite trees 
$T_0 \subseteq T_1 \subseteq T_2 \subseteq \cdots$.  Set $T_0 = \emptyset$.  
To define $T_{s+1}$, consider each string $\sigma \in 2^{< \omega}$ which has length 
$s+1$.  Let $\alpha_{\sigma}$ be the lexicographically least element of $2^{< \omega}$ which 
extends $\tau_{s,s}^{\sigma}$ and which is not in $T_s$.  Add $\alpha_{\sigma}$ and all of its 
initial segments to $T_s$.  $T_{s+1}$ is the tree formed by adding these strings when 
$\sigma$ ranges over all elements of $2^{< \omega}$ of length $s+1$.  Our desired tree is 
$T = \bigcup_s T_s$.  

\begin{lem}
\label{lem:fact1}
For each $n$ and each $\sigma \in 2^{< \omega}$ of length $n+1$, the node $(\tau_n^{\sigma})'$ 
is a branching node of $T$.
\end{lem}

\begin{proof}
Let $u_0 < u_1 < \cdots$ be the stages such that $n < u_0$ and $a(i,u_k) = a(i)$ for all $i \leq n$.  For each such 
stage, $(\tau_{n}^{\sigma})' = (\tau_{n,u_k}^{\sigma})'$ and $(\tau_{n,u_k}^{\sigma})' \subseteq \tau_{u_k,u_k}^{\xi}$ for 
all strings $\xi$ such that $\sigma' \subseteq \xi$ and $|\xi| = u_k+1$.  Therefore, both $T_{u_k}((\tau_n^{\sigma})' \ast 0)$ and 
$T_{u_k}((\tau_n^{\sigma})' \ast 1)$ gain extra elements at stage $u_k+1$.  Therefore, these 
trees are infinite and $(\tau_n^{\sigma})'$ is a branching node in $T$.
\end{proof}

\begin{lem}
\label{lem:fact2}
If $\xi$ is a branching node of $T$, then there is an $n$ such that $|\xi| = n+\sum_{i=0}^n a(i)$.
\end{lem}

\begin{proof}
Suppose $\xi \in T$ is such that there is an $n$ such that 
\begin{equation}
\label{eq:1}
n+\sum_{i=0}^{n} a(i)  < |\xi| < n+1+ \sum_{i = 0}^{n+1} a(i) .
\end{equation}
By Lemma \ref{lem:aprop}, let $u$ be a stage such that for all $s > u$ and all $i \leq n+1$, 
either $a(i,s) = a(i)$ or $a(i,s) > |\xi|$.  Fix any stage $s > u$.

We claim that there is a $j \leq n$ such that 
\begin{equation}
\label{eq:2}
j+\sum_{i=0}^j a(i,s)  < |\xi| < j+1+\sum_{i=0}^{j+1} a(i,s).
\end{equation}
The proof of the claim breaks into two cases.  If $a(i,s) = a(i)$ for all $i \leq n$, then the claim with $j=n$ 
follows from Equation (\ref{eq:1}).  Otherwise, let $j<n$ be the least number such that 
$a(j+1,s) \neq a(j+1)$.  In this case, 
\[
j+\sum_{i=0}^{j} a(i,s) = j+\sum_{i=0}^{j} a(i) < n+\sum_{i=0}^n a(i) < |\xi|.
\]
Because $a(j+1,s) > |\xi|$, we have $j+1+\sum_{i=0}^{j+1} a(i,s) > |\xi|$ and hence 
Equation \ref{eq:2} holds in this case as well.  

By Equation (\ref{eq:2}), at stage $s$ there is a unique $\sigma \in 2^{< \omega}$ with length 
$j+1$ such that $\tau_{j,s}^{\sigma} \subseteq \xi$.  Furthermore, for $a \in \{ 0,1 \}$ we have 
$\xi \subsetneq (\tau_{j+1,s}^{\sigma \ast a})'$.  That is, 
\[
0^{a(0,s)} \ast \sigma(0) \ast \cdots \ast 0^{a(j,s)} \ast \sigma(j) \subseteq \xi \subsetneq 
0^{a(0,s)} \ast \sigma(0) \ast \cdots \ast 0^{a(j,s)} \ast \sigma(j) \ast 0^{a(j+1,s)}. 
\]
It follows that at stage $s+1$, $T_s(\xi*0)$ gets a new element but $T_s(\xi*1)$ does not.  
Because this property holds for any $s > u$, $T(\xi*1) = T_u(\xi*1)$ is finite, so 
$\xi$ is not a branching node of $T$.  

To finish the proof, we need to show that if $\xi \in T$ and $|\xi| < a(0)$, then $\xi$ is not a 
branching node.  The proof of this fact is similar to (but simpler than) the argument above and we leave 
it to the reader to verify.
\end{proof}

From Lemmas \ref{lem:fact1} and \ref{lem:fact2}, we obtain the following fact.

\begin{lem}
\label{lem:fact3}
The $n^{\text{th}}$ branching level of $T$ is given by the formula $b(n) = n+\sum_{i=0}^{n} a(i) $.
\end{lem}

\begin{lem}
\label{lem:fact4}
If $T(\xi)$ is infinite, then $\xi \subseteq \tau_n^{\sigma}$ for some $n$ and $\sigma$ with 
$|\sigma| = n+1$.
\end{lem}

\begin{proof}
Suppose that $\xi \not \subseteq \tau_n^{\sigma}$ for any $n$ and $\sigma$.  We show that 
$T(\xi)$ is finite.  First, notice that $\xi$ must contain at least one value of $1$ or else 
$\xi \subseteq \tau_n^{\sigma}$ for sufficiently large $n$ by choosing $\sigma$ to contain all 
zeros.  

Second, notice that if $\xi$ does not have $0^{a(0)}$ as an initial segment, then 
this property follows trivially.  That is, fix a stage $u$ such that for all $s > u$, $a(0,s) \geq a(0)$.   
At any stage $s>u$, we add nodes only above strings $\tau_{s,s}^{\sigma}$ 
and each string $\tau_{s,s}^{\sigma}$ begins with $0^{a(0,s)}$.  Because this string is not an initial 
segment of $\xi$, $T_s(\xi)$ does not get a new element at stage $s+1$.  Therefore, 
$T_u(\xi) = T(\xi)$ and hence $T(\xi)$ is finite.  

It remains to consider the case when $0^{a(0)}$ is an initial segment of $\xi$ and $\xi$ contains 
at least one value of $1$. Let $j$ be the largest value such that there is are strings $\alpha$ 
(with $|\alpha| = j+1$) and $\mu$ such that 
\[
\xi = 0^{a(0)} \ast \alpha(0) \ast \cdots \ast 0^{a(j)}\ast \alpha(j) \ast \mu.
\]
Fix such $j$, $\alpha$ and $\mu$.  Because $\xi \not \subseteq \tau_n^{\sigma}$ for any $n$ and 
$\sigma$, the string $\mu$ must 
contain at least one value of $1$.  Write $\mu = \mu_0 \ast 1 \ast \mu_1$ where $\mu_0$ is such 
that $\mu_0(k) = 0$ for all $k < |\mu_0|$.  Because $j$ is chosen maximal, $|\mu_0| < a(j+1)$.  

Let $u$ be a stage such that for all $s > u$ and for all $i \leq j+1$, $a(i,s) = a(i)$ or $a(i,s) > |\xi|$.  
The lemma follows from the claim that $T_u(\xi) = T(\xi)$.  To prove this claim, fix any $s > u$ and we 
show that $T_s(\xi)$ does not gain a new element at stage $s+1$.  We split into two cases.  
First, suppose that for all $i \leq j$, $a(i,s) = a(i)$.  The only way for $T_s(\xi)$ to gain a new element 
at stage $s+1$ is if there is a string $\sigma$ of length $s+1$ 
such that $\xi \subseteq \tau_{s,s}^{\sigma}$.  
Because $a(i,s) = a(i)$ for $i \leq j$, this string $\sigma$ must satisfy $\sigma(i) = \alpha(i)$ for 
all $i \leq j$.  It follows that $0^{a(0)} \ast \alpha(0) \ast \cdots \ast 0^{a(j)} \ast \alpha(j) 
\ast 0^{a(j+1,s)}$ is an initial segment of $\tau_{s,s}^{\sigma}$.  However, regardless of whether 
$a(j+1,s) = a(j)$ or not, we have $|\mu_0| < a(j+1,s)$.  Hence the strings 
$0^{a(0)} \ast \alpha(0) \ast \cdots \ast 0^{a(j)} \ast \alpha(j) \ast 0^{a(j+1,s)}$ and $\xi$ are incomparable.  (The point is that 
$\xi$ contains the value 1 right after $\mu_0$ while the other string has value 0 in this position.)  
Therefore, $T_s(\xi)$ does not get a new element in this case.

The other case is when there is an $i < j$ for which $a(i,s) \neq a(i)$.  Let $k$ denote the 
least such $i$.  The argument is similar.  
$T_s(\xi)$ can gain a new element only if there is a $\sigma$ such that $\xi \subseteq 
\tau_{s,s}^{\sigma}$.  Because $a(i,s) = a(i)$ for all $i < k$, we have $\sigma(i) = \alpha(i)$ 
for $i < k$ and hence $0^{a(0)} \ast \alpha(0) \ast \cdots \ast 0^{a(k-1)} \ast \alpha(k-1) \ast 
0^{a(k,s)}$ is an initial segment of $\tau_{s,s}^{\sigma}$.  Because $a(k,s) > |\xi|$, this string is incomparable 
with $\xi$ and hence $T_s(\xi)$ does not get a new element in this case.
\end{proof}

\begin{lem}
\label{lem:fact5}
The tree $T$ has no isolated paths.  
\end{lem}

\begin{proof}
This lemma follows immediately from Lemmas \ref{lem:fact4} and \ref{lem:fact1}.  
\end{proof}

\begin{lem}
\label{lem:fact6}
If $\delta:T \rightarrow T$ is a nontrivial self-embedding, then there is a string $\xi$ such that 
$|\delta(\xi)| > |\xi|$.
\end{lem}

\begin{proof}
Suppose there is no such string $\xi$.  Because $|\xi| \leq |\delta(\xi)|$ for any self-embedding 
$\delta$ and because $T$ is binary branching, it follows that for each $n$, $\delta$ restricted to the strings of length $n$ in $T$ 
is a permutation.  Therefore $\delta$ is onto and hence is not nontrivial. 
\end{proof}

\begin{lem}
\label{lem:fact7}
If $\delta:T \rightarrow T$ is a nontrivial self-embedding then there is a $k \in \omega$ and 
a node $\xi$ such that $\xi \subsetneq \delta^k(\xi)$.  
\end{lem}

\begin{proof}
By Lemma \ref{lem:fact6}, let $\mu_0$ be a node such that $|\mu_0| < |\delta(\mu_0)|$.  
If $\mu_0 \subseteq \delta(\mu_0)$ then $\mu_0 \subsetneq \delta(\mu_0)$ and we can 
let $\xi = \mu_0$ and $k=1$ to verify the lemma.  Otherwise, assume that $\mu_0 \not 
\subseteq \delta(\mu_0)$.  Let $\mu_1$ be such that $|\mu_1| = |\mu_0|$ and 
$\mu_1 \subseteq \delta(\mu_0)$.  Notice that $\mu_1 \neq \mu_0$ and 
$\delta(\mu_0) \neq \mu_0$.  

We proceed by induction.   Assume that $n \geq 1$ and we have defined a sequence of 
pairwise distinct nodes 
$\mu_0, \mu_1, \ldots, \mu_n$ such that $|\mu_i| = |\mu_0|$ for all $i \leq n$ and 
$\mu_{i+1} \subseteq \delta(\mu_i)$ and $\mu_i \neq \delta(\mu_i)$ for all $i < n$.  (The last 
two sentences of the previous paragraph establish the required properties when $n=1$.)  

We claim that in this situation, $\mu_n \neq \delta(\mu_n)$.  Suppose that $\mu_n = 
\delta(\mu_n)$.  Because $|\mu_n| = |\mu_{n-1}|$ and $\mu_n \neq \mu_{n-1}$, 
$\mu_n$ and $\mu_{n-1}$ are incomparable nodes.  However, $\delta(\mu_n) = \mu_n 
\subseteq \delta(\mu_{n-1})$.  Therefore $\delta(\mu_n)$ and $\delta(\mu_{n-1})$ are comparable 
contradicting the fact that $\delta$ is a self-embedding.  

Next, we let $\mu_{n+1}$ be such that $|\mu_{n+1}| = |\mu_0|$ and $\mu_{n+1} \subseteq 
\delta(\mu_n)$.  We claim that if $\mu_{n+1} = \mu_i$ for some $i \leq n$, then the conclusion of 
the lemma is true.  Otherwise, if $\mu_{n+1} \neq \mu_i$ for all $i \leq n$, then we add $\mu_{n+1}$ to the list of pairwise distinct 
nodes above and continue by induction.  
Because there are only finitely many nodes at level $|\mu_0|$, we must eventually find an $n$ such 
that $\mu_{n+1} = \mu_i$ for some $i \leq n$.  Hence, the lemma follows from the claim in this 
paragraph.  

Suppose that $\mu_{n+1} = \mu_i$ for some $i \leq n$ and let 
$l \geq 1$ be such that $i = (n+1) - l$.  In this 
situation we have $\mu_i = \mu_{n+1} \subseteq \delta^l(\mu_i)$.  We claim that 
$\mu_i \neq \delta^l(\mu_i)$ (and hence we have established the lemma with $\xi = \mu_i$ 
and $k = l$).  We break into three cases.  

\begin{description}

\item[Case 1.] $i=0$.  In this case, we have $\mu_0 \subseteq 
\delta^l(\mu_0)$.  But, $|\mu_0| < |\delta(\mu_0)|$ implies $|\mu_0| < |\delta^l(\mu_0)|$ 
so we have $\mu_0 \subsetneq \delta^l(\mu_0)$ as required.

\item[Case 2.] $l=1$.  In this case, we have $\mu_i \subseteq \delta(\mu_i)$.  Because $i \leq n$, we know 
$\mu_i \neq \delta(\mu_i)$, so $\mu_i \subsetneq \delta(\mu_i)$ as required.

\item[Case 3.]  $i > 0$ and $l > 1$.  For a contradiction, assume that 
$\mu_i = \delta^l(\mu_i)$.  We have $\mu_i = \mu_{n+1} \subseteq \delta(\mu_n)$ and 
$\mu_i \subseteq \delta(\mu_{i-1})$.  By our induction hypothesis, $\mu_n$ and $\mu_{i-1}$ 
are incomparable nodes.  Furthermore, we have 
\[
\mu_i = \mu_{n+1} \subseteq \delta(\mu_n) \subseteq \delta^2(\mu_{n-1}) \subseteq 
\cdots \subseteq \delta^l(\mu_i) = \mu_i.
\]
Therefore, $\delta(\mu_n) = \mu_i$ so $\delta(\mu_n)$ and $\delta(\mu_{i-1})$ are comparable 
nodes, violating the fact that $\delta$ is a self-embedding.

\end{description}
\end{proof}

For any nontrivial self-embedding $\delta:T \rightarrow T$, we can fix $k$ and $\xi$ as in 
Lemma \ref{lem:fact7} and let $\gamma = \delta^k:T \rightarrow T$.  $\gamma$ is a nontrivial 
self-embedding of $T$ such that $\xi \subsetneq \gamma(\xi) \subsetneq \gamma^2(\xi) \subsetneq 
\cdots$.  It will also be useful to consider the nontrivial self-embedding $\iota:T \rightarrow T$ 
given by $\iota = \gamma^2$.  Notice that both $\iota$ and $\gamma$ are obtained from 
$\delta$ by finitely many parameters.

\begin{lem}
\label{lem:fact8}
Let $\delta$ be any nontrivial self-embedding $\delta:T \rightarrow T$ and let $\gamma$ and 
$\iota$ be defined from $\delta$ as above.  There are nodes $\alpha$ and $\beta_0$ such that 
$\alpha \subsetneq \beta_0 \subsetneq \iota(\alpha)$ and $\beta_0$ is a branching node.
\end{lem}

\begin{proof}
Fix $\xi$ as in the paragraph before this lemma.  Because $\xi, \gamma(\xi), \gamma^2(\xi), 
\ldots$ traces out a path in $T$ and because $T$ has no isolated paths, there must be a 
$j \geq 1$ and a branching node $\beta_0$ such that $\gamma^j(\xi) \subseteq \beta_0 
\subsetneq \gamma^{j+1}(\xi)$.  Let $\alpha = \gamma^{j-1}(\xi)$.  Because 
$\alpha = \gamma^{j-1}(\xi) \subsetneq \gamma^j(\xi) \subseteq \beta_0$, we have 
$\alpha \subsetneq \beta_0$.  Because $\iota = \gamma^2$, we have $\beta_0 \subsetneq 
\gamma^{j+1}(\xi) = \gamma^2(\gamma^{j-1}(\xi)) = \iota(\alpha)$.  
\end{proof}

\begin{lem}
\label{lem:fact8+}
Let $\iota:T \rightarrow T$ be a nontrivial self-embedding for which there are nodes 
$\alpha$ and $\beta_0$ such that $\alpha \subsetneq \beta_0 \subsetneq \iota(\alpha)$ 
and $\beta_0$ is a branching node.  Then there is a branching node $\beta_1$ such that 
$\iota(\alpha) \subsetneq \beta_1 \subsetneq \iota^2(\alpha)$.  
\end{lem}

\begin{proof}
Fix $\alpha$ and $\beta_0$.  Because $\alpha \subsetneq \beta_0 \subsetneq \iota(\alpha)$, 
we have 
\begin{gather*}
\alpha \subsetneq \beta_0 \subsetneq \iota(\alpha) \subsetneq \iota(\beta_0) \subsetneq 
\iota(\beta_0 \ast 0) \\ 
\alpha \subsetneq \beta_0 \subsetneq \iota(\alpha) \subsetneq \iota(\beta_0) \subsetneq 
\iota(\beta_0 \ast 1).
\end{gather*}  
Let $\beta_1$ be the infimum of $\iota(\beta_0 \ast 0)$ and $\iota(\beta_0 \ast 1)$.  Because 
these two nodes are incomparable, $\beta_1$ is strictly contained in both of them.  From 
the offset containments above, it is clear than $\beta_0 \subsetneq \iota(\alpha) \subsetneq \beta_1$.  
Let $i_0 \in \{ 0,1 \}$ be such that $\beta_0 \ast i_0 \subseteq \iota(\alpha)$.  Because 
$\iota(\beta_0 \ast i_0) \subseteq \iota^2(\alpha)$ and $\beta_1 \subsetneq \iota(\beta_0 
\ast i_0)$ we have $\beta_1 \subsetneq \iota^2(\alpha)$.

Finally, because $\beta_0$ is a branching node, both $T(\iota(\beta_0*0))$ and $T(\iota(\beta_0*1))$ 
are infinite.  Therefore, the infimum of $\iota(\beta_0*0)$ and $\iota(\beta_0*1)$ (which is $\beta_1$) is a branching node.  
\end{proof}

\begin{lem}
\label{lem:fact9}
Let $\delta:T \rightarrow T$ be any nontrivial self-embedding.  There is a nontrivial self-embedding 
$\iota:T \rightarrow T$ (defined from $\delta$ together with finitely many parameters) and a node $\alpha$ such that 
the sequence $c(n) = |\iota^{n+1}(\alpha)|$ dominates the branching level function $b(n)$ of $T$ (see Lemma \ref{lem:fact3}).   
\end{lem}

\begin{proof}
Define $\iota$ from $\delta$ as above and let $\alpha$ and $\beta_0$ be as in 
Lemma \ref{lem:fact8}.  Applying Lemma \ref{lem:fact8+} inductively, we obtain a sequence 
of branching node $\beta_0 \subsetneq \beta_1 \subsetneq \beta_2 \subsetneq \cdots$ such 
that $\iota^n(\alpha) \subsetneq \beta_n \subsetneq \iota^{n+1}(\alpha)$.  (For $n=0$, we define $\iota^0(\alpha) = \alpha$.)  
Therefore, there are at least $n$ many branching levels below $|\iota^{n+1}(\alpha)|$.  
\end{proof}

To prove Theorem \ref{thm:doublecode}, we need to transform the c.e.~subtree $T \subseteq 
2^{< \omega}$ into a computable tree $S$.  This transformation is easily done in a general 
setting.

\begin{lem}
\label{lem:transform}
For any c.e.~subtree $\hat{T} \subseteq 2^{< \omega}$,  there is a computable tree $\hat{S}$ such that 
$\hat{S} \cong \hat{T}$.  Furthermore, we can assume that the successor relation is computable 
in $\hat{S}$.  
\end{lem}

\begin{proof}
If $\hat{T}$ is finite, this lemma follows trivially.  Assume $\hat{T}$ is infinite and $\hat{T}$ is the range of 
the total computable 1-1 function $\varphi_e$.  Let $\hat{S}$ have domain $\omega$ and 
let $\leq_{\hat{S}}$ be defined by $n \leq_{\hat{S}} m \Leftrightarrow \varphi_e(n) \subseteq 
\varphi_e(m)$.  Then $\varphi_e$ is an isomorphism from $(\hat{S},\leq_{\hat{S}})$ to 
$(\hat{T}, \subseteq)$ as required.  Furthermore, $m$ is a successor of $n$ in $\hat{S}$ 
if and only if $\varphi_e(m)' = \varphi_e(n)$.  
\end{proof}

We now present the proof of Theorem \ref{thm:doublecode}.  Let $T$ be the c.e.~subtree of 
$2^{< \omega}$ we have constructed and let $S \cong T$ be the computable tree with 
a computable successor relation given by Lemma \ref{lem:transform}.  Let $b$ denote 
the branching level function for $S$ (which is the same as the branching level function for 
$T$ since $S \cong T$).  By Lemma \ref{lem:fact3}, $b$ dominates $a$ and hence 
by Lemma \ref{lem:aprop}, $0'' \leq_T b$.  Fix any nontrivial self-embedding 
$\delta:S \rightarrow S$.  By Lemma \ref{lem:fact9}, there is a nontrivial self-embedding 
$\iota:S \rightarrow S$ (defined from finitely many parameters) and a node 
$\alpha$ such that the function $c(n) = \text{ht}(\iota^{n+1}(\alpha))$ dominates $b$ 
(and hence by Lemma \ref{lem:aprop}, $0'' \leq_T c$).  Because we only need 
finitely many parameters to obtain $\iota$ from $\delta$, we have $\iota \leq _T \delta$.  
Furthermore, because the successor function is computable in $S$, we can determine 
the height of any node in $S$.  Therefore, $0'' \leq_T c \leq_T \iota \leq_T \delta$ as 
required.  This completes the proof of Theorem \ref{thm:doublecode}.

Because the coding in the proof of Theorem \ref{thm:doublecode} is done with the branching levels of $S$ and these levels 
are invariant under isomorphisms, we also obtain a result concerning the existence of nontrivial self-embeddings of computable type 3 
trees up to isomorphism.  

\begin{thm}
\label{iso0''}
There is a computable type 3 tree $S$ such that for any computable $\hat{S} \cong S$ and any nontrivial 
self-embedding $\delta: \hat{S} \rightarrow \hat{S}$, $0'' \leq_T 0' \oplus \delta$.  In particular, 
$\hat{S}$ does not have any $\Delta^0_2$ nontrivial self-embeddings and $\hat{S}$ does 
not have any nontrivial self-embeddings which are strictly between $0'$ and $0''$.   
\end{thm}

\begin{proof}
Fix a computable tree $S \cong T$ where $T$ is the c.e.~subtree 
$T \subseteq 2^{< \omega}$ constructed above.  Fix any $\hat{S} \cong S$ 
and any nontrivial self-embedding $\delta$ of $\hat{S}$.  Let $\iota$ and $c$ be the 
functions given by Lemma \ref{lem:fact9} for $\hat{S}$ and $\delta$.  The only change from 
the proof of Theorem \ref{thm:doublecode} is that we do not know that the successor relation 
in $\hat{S}$ is computable.  However, since $0' \oplus \delta$ can compute both $\iota$ and the 
successor relation in $\hat{S}$, we have $0'' \leq_T c \leq_T 0' \oplus \delta$.
\end{proof}

\section{Chains and antichains}
\label{sec:cac}

The fact that any infinite partial order must have either an infinite chain or an infinite antichain is a simple application of Ramsey's Theorem for 
pairs and two colors.  Herrmann \cite{her:01} examined the effective content of this result and proved the following theorem.

\begin{thm}[Herrmann \cite{her:01}] 
\label{thm:herrmann} 
If $P$ is an infinite computable partial order, then $P$ has either an infinite $\Delta^0_2$ 
chain or an infinite $\Pi^0_2$ antichain.  In addition, there is an infinite computable partial order 
which has no infinite $\Sigma^0_2$ chains or antichains.
\end{thm}

In this section, we consider this result in the context of trees rather than general partial orders and 
we show that for trees these results can be improved by exactly one quantifier.  

\begin{thm}
Let $T$ be an infinite computable tree.  $T$ has either an infinite computable chain or an 
infinite $\Pi^0_1$ antichain.
\end{thm}

\begin{proof}
If $T$ has infinitely many leaves, then the set of leaves is an infinite $\Pi^0_1$ antichain.  
Otherwise, $T$ must have a node $x$ such that $T(x)$ is infinite and contains no leaves.  In this 
case, let $x_0 = x$ and $x_{i+1}$ be the $\leq_{\mathbb{N}}$ least element of $T$ which 
satisfies $x_i \prec x_{i+1}$.  The sequence $x_0, x_1, \ldots$ gives an infinite computable chain.
\end{proof}  

\begin{thm}
\label{thm:best}
There is an infinite binary branching computable tree such that $T$ has no infinite c.e.~chains or  antichains.
\end{thm}

\begin{proof}
We build $(T, \preceq)$ to meet the following requirements.   
\begin{gather*}
R_{2e}: \, W_e \, \text{is not an infinite chain} \\
R_{2e+1}: \, W_e \, \text{is not an infinite antichain}
\end{gather*}

We build $T$ in stages beginning with $T_0 = \{ \lambda \}$.  Throughout the construction, 
we maintain the property that each node $x$ is either currently a leaf or else has 
exactly two successors.  Each requirement $R_i$ keeps a parameter $r_i$ 
such that any node $x$ added to $T$ by a lower priority requirement after $r_i$ is defined 
satisfies $r_i \preceq x$.  For uniformity of notation, we set $r_{-1} = \lambda$.  If a strategy 
is initialized, then all of its parameters become undefined.  Any parameter not explicitly redefined or  
undefined by initialization retains its value.  If a requirement ends the current stage, then 
it initializes all lower priority requirements.  The action for $R_{2e}$ at stage $s$ is as follows.
\begin{enumerate}
\item If $s$ is the first stage at which $R_{2e}$ is eligible to act or if $R_{2e}$ has been initialized 
since it was last eligible to act, let $a$ be such that $r_{2e-1} \preceq a$ and 
$a$ is a leaf in $T_s$.  Add new nodes $b$ and $c$ to $T_s$ 
as immediate successors of $a$.  Set $r_{2e} = b$ and end the stage.  
\item If $r_{2e}$ is defined but $R_{2e}$ has not succeeded yet, then check whether there 
is a node $x \in T_s$ such that $r_{2e} \preceq x$ and $x \in W_{e,s}$.  If not, then 
let $R_{2e+1}$ act.  If so, then let $z$ denote the immediate predecessor 
of $x$.  Since $z$ is not a leaf, it has two immediate successors.  Let $y$ denote the successor 
of $z$ which is not equal to $x$.  Redefine $r_{2e} = y$ and end the stage.  
We say that $R_{2e}$ has succeeded. 
\item If $R_{2e}$ has succeeded, then let $R_{2e+1}$ act.   
\end{enumerate}
The action for $R_{2e+1}$ at stage $s$ is as follows.
\begin{enumerate}
\item If $s$ is the first stage at which $R_{2e+1}$ is eligible to act or if $R_{2e+1}$ 
has been initialized since it was last eligible to act, define $r_{2e+1} = r_{2e}$.  End the 
stage.
\item If $r_{2e+1}$ is defined and $R_{2e+1}$ has not succeeded yet, check whether 
there is a node $x \in T_s$ such that $r_{2e+1} \preceq x$ and $x \in W_e$.  If not, then 
let $R_{2e+2}$ act.  If so, then redefine $r_{2e+1} = x$ and end the stage.  
We say $R_{2e+1}$ succeeds.
\item If $R_{2e+1}$ has succeeded, then let $R_{2e+2}$ act.
\end{enumerate}

This argument is finite injury so each parameter reaches a limit.  Because 
nodes are added to $T$ only in Step 1 of the $R_{2e}$ action, $T$ has the property 
that at each stage, each node is either currently a leaf or has exactly two successors.  

To see that $R_{2e}$ is met, let $s$ be the least stage such that $R_{2e}$ is never initialized after 
stage $s$.  The parameter $r_{2e}$ is defined at stage $s$ and can only change values after 
stage $s$ if $R_{2e}$ changes the value in Step 2 of its action.  

There are two cases to consider.  First, suppose there is a stage $t > s$ and a node $x \in T_t$ such that 
$r_{2e} \preceq x$ and $x \in W_{e,t}$.  In this case, $r_{2e}$ is redefined so that 
$r_{2e}$ is incompatible with $x$.  Because $r_{2e}$ is not changed again and because 
no strategy of higher priority than $R_{2e}$ adds elements to $T$ after stage $t$, there are 
only finitely many elements in $T$ which are not above this final value of $r_{2e}$.  Therefore, 
$x$ cannot be part of an infinite chain and $R_{2e}$ is met.

Second, suppose there is no such stage $t$ and node $x$.  In this case, $r_{2e}$ 
has reached its limit at stage $s$ and every node added to $T$ after stage $s$ is added 
above $r_{2e}$.  Because there are only finitely many nodes in $T$ which are not above 
$r_{2e}$, there cannot be an infinite chain which is disjoint from $T(r_{2e})$.  Therefore, 
$R_{2e}$ is met.  

The argument that $R_{2e+1}$ is met is quite similar.  Let $s$ be the least stage such that 
$R_{2e+1}$ is never initialized after $s$ and let $r_{2e+1}$ denote the value of the parameter 
at stage $s$.  If there is no node $x \in W_e$ such that $r_{2e+1} \preceq x$, then 
$R_{2e+1}$ never changes the value of $r_{2e+1}$ and there are only finitely many 
nodes of $T$ which are not above $r_{2e+1}$.  Any infinite antichain must intersect 
$T(r_{2e+1})$, so $R_{2e+1}$ is met.  

If there is a node $x \in W_e$ such that $r_{2e+1} \preceq x$, then let $x$ be the first such node seen by 
$R_{2e+1}$ after stage $s$.  At this point, $r_{2e+1}$ is redefined to be equal to $x$, so there 
are only finitely many nodes in $T$ which are not comparable to $x$.  Hence, no set containing 
$x$ can be an infinite antichain.  Therefore, $R_{2e+1}$ is won.  
\end{proof}

\begin{thm}
\label{thm:low}
Let $L$ be any low set.  There is an infinite binary branching computable tree $T$ such that $T$ has no 
infinite chains or antichains computable from $L$.
\end{thm}

\begin{proof}
We need to meet the following requirements.
\begin{gather*}
R_{2e}: \, \varphi_e^L \, \text{is not an infinite chain} \\
R_{2e+1}: \, \varphi_e^L \, \text{is not an infinite antichain}
\end{gather*}

As in the proof of Theorem \ref{thm:best}, we build $T$ in stages 
and maintain the property that each node $x$ is either currently a leaf or 
else has exactly two immediate successors.  Each requirement $R_i$ keeps a 
parameter $r_i$ as before.  The main change in this construction is that the value of 
$r_i$ can change more than once (but still only finitely often) after the last time $R_i$ is initialized.  
As before, whenever a requirement ends a stage, it initializes all lower priority requirements.  

$R_{2e}$ keeps three parameters: $r_{2e}$,  
$\hat{r}_{2e}$ and $x_{2e}$.  The $r_{2e}$ parameter is used as before to force lower 
priority requirements to 
work above $r_{2e}$.  The $\hat{r}_{2e}$ parameter is used to store an ``old value'' of 
$r_{2e}$ in case our approximations to computations from $L$ change and we need to 
revert back to an earlier 
situation and wait for reconvergence.  The $x_{2e}$ parameter will be explained when it 
appears in the construction below.  

When $R_{2e}$ first acts or if $R_{2e}$ has been initialized since its last action, it 
lets $a$ be a node such that $r_{2e-1} \preceq a$ and $a$ is currently a leaf.  It adds 
two new nodes $b$ and $c$ as immediate successors of $a$ in $T$, defines 
$\hat{r}_{2e} = r_{2e} = b$ and ends the stage.  At future stages $s$, $R_{2e}$ requires lower 
priority strategies to work above $r_{2e}$ and it tries to decide 
whether $\exists x \, \exists t \, ( r_{2e} \preceq x \wedge \varphi_{e,t}^L(x) = 1)$.  This predicate is 
$\Sigma^L_1$, so it is computable from $L'$ and hence from $0'$ (since $L$ is low).  
Fix the $\Delta^0_2$ predicate $P(e,k)$ defined by 
\[
P(e,k) \, \Leftrightarrow \, \exists x \, \exists t \, ( k \preceq x \wedge \varphi_{e,t}^L(x) = 1).
\]
Let $P(e,k,s)$ be a computable approximation such that 
$P(e,k) = \lim_{s} P(e,k,s)$.  

At stage $s$, $R_{2e}$ checks whether $P(e,\hat{r}_{2e},s) = 1$.  (Notice that 
$\hat{r}_{2e} = r_{2e}$ at this point, so $R_{2e}$ is really checking whether 
$P(e,r_{2e},s) = 1$.)  If not, then $R_{2e}$ has no need to diagonalize and it lets 
$R_{2e+1}$ act.  If so, then $R_{2e}$ wants to find the least witness $x$ 
which appears to satisfy this existential statement.  We define a second $\Delta^0_2$ 
predicate $Q(x,u)$ by
\[
Q(x,u) \Leftrightarrow \forall t ( t \geq u \rightarrow \varphi_{e,t}^{L}(x) = 1)
\]
and fix a computable approximation $Q(x,u,s)$ such that $Q(x,u) = \lim_{s}Q(x,u,s)$.  (The 
predicate $Q(x,u)$ is computable from $L'$, so it is $\Delta^0_2$ because $L$ is low.)  
$R_{2e}$ looks for the least $x$ such that $\hat{r}_{2e} = r_{2e} \preceq x$ 
and $Q(x,s,s)$.  If there is no such $x$ then $R_{2e}$ lets $R_{2e+1}$ act and waits to check again 
at the next stage.  Eventually, it must find an $x$ and $s$ for which $Q(x,s,s)$.  
(Of course, if $P(e,\hat{r}_{2e},s)$ changes from value 1 to value 0 while $R_{2e}$ is waiting for 
such an $x$, it ends the stage and returns to waiting for $P(e,\hat{r}_{2e},s)$ to have value 1.)  

When $R_{2e}$ finds such an $x$, it defines its third parameter $x_{2e} = x$.  Let $z$ be 
the immediate predecessor of $x_{2e}$ and let $y$ 
be the successor $z$ which is not equal to $x_{2e}$.  $R_{2e}$ sets $r_{2e} = y$ but leaves the 
value of $\hat{r}_{2e}$ unchanged.  That is, $\hat{r}_{2e}$ retains the ``old value'' of $r_{2e}$.  
$R_{2e}$ ends the stage (and hence initializes the lower priority strategies so that they will work 
above the new value of $r_{2e}$ in  the future).  

If $P(e,\hat{r}_{2e})$ really holds and $x_{2e}$ really is a correct witness for this existential statement, 
then we have successfully diagonalized.  However, it is possible that either 
$P(e,\hat{r}_{2e})$ does not hold or that $x_{2e}$ is not a correct witness.  Therefore, at each future 
stage $s$, $R_{2e}$ continues to check whether $P(e,\hat{r}_{2e},s) = 1$.  If this 
value ever changes to $0$, then $R_{2e}$ redefines $r_{2e}$ to have value 
$r_{2e} = \hat{r}_{2e}$, cancels its parameter $x_{2e}$, ends the stage 
and returns to waiting for $P(e,\hat{r}_{2e},s) = 1$.  If 
$P(e,\hat{r}_{2e},s)$ retains its value of $1$, then $R_{2e}$ checks whether 
$Q(x_{2e},s,s)$ still gives the value 1.  If so, then $R_{2e}$ continues to believe it has correctly 
diagonalized and lets $R_{2e+1}$ act.  If $Q(x_{2e},s,s) = 0$ at some future stage $s$ 
(while $P(e,\hat{r}_{2e},s) = 1$), then $R_{2e}$ cancels the parameter $x_{2e}$, redefines $r_{2e}$ 
to have value $r_{2e} = \hat{r}_{2e}$, ends the stage and returns to looking for the 
least $x$ such that $Q(x,s,s) = 1$.  

To understand why this strategy eventually succeeds, let $u$ be a stage such that 
$R_{2e}$ is never initialized after $u$.  At stage $u$, $R_{2e}$ defines $r_{2e}$ 
and $\hat{r}_{2e}$ and it will never change the value of $\hat{r}_{2e}$ again.  (The entire 
construction is finite injury so there is such a stage.)  Because 
$P(e,\hat{r}_{2e})$ is a $\Delta^0_2$ predicate, there is a $t \geq u$ such that 
for all $s \geq t$, $P(e,\hat{r}_{2e},s)$ is either constantly 0 or constantly 1.  

If $P(e, \hat{r}_{2e},s)$ is eventually constantly 0, then $r_{2e}$ will eventually be set 
permanently equal to $\hat{r}_{2e}$.  From this stage on, all nodes added to $T$ are 
above $\hat{r}_{2e}$.  Because $P(e,\hat{r}_{2e})$ does not hold, $\varphi_e^L$ 
does not place any elements from $T(\hat{r}_{2e})$ into its chain.  Because 
there are only finitely many elements of $T$ outside of $T(\hat{r}_{2e})$, 
$\varphi_e^L$ cannot compute an infinite chain and $R_{2e}$ is met.  

If $P(e,\hat{r}_{2e},s)$ is eventually constantly 1, then there is an $x$ such that $\hat{r}_{2e} 
\preceq x$ and $x$ is a witness to the existential statement $P(e,\hat{r}_{2e})$.  
Because $Q(x,u,s)$ is a $\Delta^0_2$ predicate and we look for the  
least witness $x$, $R_{2e}$ eventually defines 
$x_{2e}$ such that $Q(x_{2e},s,s)$ has reached its limit of 1.  Both $r_{2e}$ and $x_{2e}$ 
have reached their limits at this stage.  
After this stage, all elements added to $T$ are above $r_{2e}$ and hence are 
incomparable with $x_{2e}$.  Because $\varphi_e^L(x_{2e}) = 1$, $\varphi_e^L$ 
cannot compute an infinite chain in $T$ so $R_{2e}$ is met.  

In either case, notice that $\hat{r}_{2e}$ and $r_{2e}$ reach limits and that 
$x_{2e}$ either reaches a limit or there is a stage after which it is never defined.  Therefore, 
$R_{2e}$ only initializes lower priority strategies finitely often.  

The strategy to meet requirement $R_{2e+1}$ is similar.  $R_{2e+1}$ also keeps three 
parameters $r_{2e+1}$, $\hat{r}_{2e+1}$ and $x_{2e+1}$.  When it first acts (or after it 
has been initialized), $R_{2e+1}$ sets $\hat{r}_{2e+1} = r_{2e+1} = r_{2e}$ and ends the stage.

At future stages, $R_{2e+1}$ checks whether $P(e,\hat{r}_{2e+1},s) = 1$.  If not, it lets 
$R_{2e+2}$ act.  If so, it looks for the least $x$ such that 
$Q(x,s,s) = 1$.  If there is no such $x$, it lets $R_{2e+2}$ act next.  If there is such an $x$, 
it sets $x_{2e+1} = x$, redefines $r_{2e+1}$ so that $r_{2e+1} = x_{2e+1}$ 
and ends the stage.  (As above, it leaves 
$\hat{r}_{2e+1}$ unchanged to mark the ``old value'' of $r_{2e+1}$.  If 
$P(e,\hat{r}_{2e+1},s)$ changes values from 1 to 0 while $R_{2e+1}$ is waiting for such 
an $x$, it ends the stage and returns to waiting for $P(e,\hat{r}_{2e+1},s)$ to equal 1.)  

Once $x_{2e+1}$ is defined, $R_{2e+1}$ continues to check whether $P(e,\hat{r}_{2e+1},s) = 1$.  
If this value ever changes to 0, it cancels $x_{2e+1}$, redefines $r_{2e+1}$ so that 
$r_{2e+1} = \hat{r}_{2e+1}$, ends the stage and returns to waiting for $P(e,\hat{r}_{2e+1},s)$ to 
equal 1.  As long as $P(e,\hat{r}_{2e+1},s)$ remains equal to 1, $R_{2e+1}$ checks whether 
$Q(x_{2e+1},s,s)$ continues to equal 1.  As long as it does, $R_{2e+1}$ lets $R_{2e+2}$ act.  
If $Q(x_{2e+1},s,s)$ changes values to 0, then $R_{2e+1}$ cancels $x_{2e+1}$, redefines 
$r_{2e+1}$ to have value $r_{2e+1} = \hat{r}_{2e+1}$, ends the stage and 
and returns to looking for the least $x$ such that $Q(x,s,s) = 1$.  

The analysis that $R_{2e+1}$ eventually succeeds and that it initializes the lower priority 
requirements only finitely often is similar to the analysis given for $R_{2e}$.  The details are left 
to the reader.
\end{proof}

There is no need to restrict ourselves to a single low set $L$ in the proof of Theorem 
\ref{thm:low}.  That is, essentially the same proof (with a little extra bookkeeping in the indices) 
shows that if $L_i$ (for $i \in \mathbb{N}$) is a sequence of uniformly low, uniformly 
$\Delta^0_2$ sets, then there is an infinite binary branching computable tree $T$ such that 
$T$ has no infinite chains and no infinite antichains computable from any of the $L_i$ sets.  
By Jockusch and Soare \cite{joc:72} and Simpson \cite{simp:book}, there is an 
$\omega$-model $\mathcal{M}$ of $\text{WKL}_0$ such that the second order part of 
$\mathcal{M}$ consists of all the sets in the Turing ideal generated by a sequence 
$L_0 \leq_T L_1 \leq_T \cdots$ of uniformly low, uniformly $\Delta^0_2$ sets.  Thus, we obtain 
the following corollary.

\begin{cor}
\label{cor:wkl}
$\text{WKL}_0$ is not strong enough to prove that every infinite binary branching tree has 
either an infinite chain or an infinite antichain.
\end{cor}


\begin{thebibliography}{99}

\bibitem{cen:99} D.~Cenzer, $\Pi^0_1$ classes in recursion theory, in \textit{Handbook of Computability Theory}, ed.~Griffor, Elsevier, 
Amsterdam, 1999, 37-88.

\bibitem{cen:98} D.~Cenzer and J.B.~Remmel, $\Pi^0_1$ classes in mathematics, in \textit{Handbook of Recursive Mathematics, 
Volume 2}, ed.~Ershov, Goncharov, Nerode and Remmel, Elsevier, Amsterdam, 1998, 623-822.

\bibitem{cor:85} E.~Corominas, On better quasi-ordering countable trees, \textit{Discrete Mathematics} 
\textbf{53} (1985), 35-53.  

\bibitem{dow:ta} R.~Downey, C.~Jockusch and J.S.~Miller, On self-embeddings of computable linear orderings, to appear in \textit{Annals of Pure 
and Applied Logic}.  

\bibitem{dow:99}  R.~Downey and S.~Lempp, On the proof-theoretic strength of the Dushnik--Miller Theorem for countable linear orderings, 
in \textit{Recursion Theory and Complexity}, ed.~Arslanov and Lempp, de Gruyter, 1999, 55-58.  

\bibitem{dus:40}  B.~Dushnik and E.~Miller, Concerning similarity transformations of linearly ordered sets, 
\textit{Bulletin of the American Mathematical Society} \textbf{46} (1940), 322-326.  

\bibitem{her:01} E.~Herrmann, Infinite chains and antichains in computable partial orderings, 
\textit{Journal of Symbolic Logic} \textbf{66} (2001), 923-934.

\bibitem{joc:72} C.G.~Jockusch Jr. and R.I.~Soare, $\Pi^0_1$ classes and degrees of theories, 
\textit{Transactions of the American Mathematical Society} \textbf{173} (1972), 33-56.

\bibitem{kru:60}  J.B.~Kruskal, Well quasi-ordering, the tree theorem and V\'{a}zsonyi's conjecture, 
\textit{Transactions of the American Mathematical Society} \textbf{95} (1960), 210-225.

\bibitem{lem:05} S.~Lempp, C.~McCoy, R.~Miller and R.~Solomon, Computable categoricity of trees of finite height, 
\textit{Journal of Symbolic Logic} \textbf{70}(2005), 151-215.

\bibitem{mil:05} R.~Miller, The computable dimension of trees of infinite height, 
\textit{Journal of Symbolic Logic} \textbf{70}(2005), 111-141.  

\bibitem{ric:81} L.J.~Richter, Degrees of structures, \textit{Journal of Symbolic Logic} \textbf{46}(1981), 723-731.

\bibitem{ros:89} D.~Ross, Tree self-embeddings, \textit{Canadian Mathematical Bulletin} 
\textbf{32} (1989), 30-33.

\bibitem{simp:book} S.G.~Simpson, \textit{Subsystems of second order arithmetic}, Springer-Verlag, Berlin, 1999.

\bibitem{soa:book} R.I.~Soare, \textit{Recursively Enumerable Sets and Degrees}, Springer-Verlag, Berlin, 1987.

\end{thebibliography}
  
\end{document}